\documentclass{amsart}

\usepackage{amsmath, amsthm, amssymb, amscd}

\usepackage{enumerate}      
\usepackage{stmaryrd}       
\usepackage{xspace}         
\usepackage{verbatim}       
\usepackage{url}            

\usepackage[dvipsnames]{xcolor}

\usepackage{tikz}
\usepackage{tikz-cd}
\usetikzlibrary{
  arrows.meta,               
  decorations.pathmorphing,  
  decorations.markings,      
  positioning                
}
\usetikzlibrary{decorations.pathreplacing}
\usetikzlibrary{patterns}
\usepackage[colorlinks=true,
            linkcolor=blue,
            citecolor=ForestGreen,
            urlcolor=BrickRed,
            bookmarksopen=true,
            bookmarksdepth=3]{hyperref}
\usepackage[active]{srcltx}
\usepackage[all]{xypic}
\SelectTips{cm}{}

{

   \newtheorem{theorem}[subsubsection]{Theorem}
      \newtheorem*{theorem*}{Theorem}
   \newtheorem{proposition}[subsubsection]{Proposition}

   \newtheorem{lemma}[subsubsection]{Lemma}

   \newtheorem{corollary}[subsubsection]{Corollary}
   
   \newtheorem*{conjecture*}{Conjecture}

}
{\theoremstyle{definition}
          \newtheorem*{exercise*}{Exercise}
   
   \newtheorem{example}[subsubsection]{Example}
   \newtheorem*{example*}{Example}
   \newtheorem{definition}[subsubsection]{Definition}
   
   \newtheorem*{definition*}{Definition}
   
   \newtheorem{remark}[subsubsection]{Remark}

}
\newcommand{\QQ}{{\mathbb{Q}}}
\newcommand{\NN}{{\mathbb{N}}}

\newcommand{\PP}{{\mathbb{P}}}
\newcommand{\ZZ}{{\mathbb{Z}}}

\renewcommand{\AA}{{\mathbb{A}}}

\newcommand{\bA}{{\mathbf{A}}}

\newcommand{\cA}{{\mathcal A}}

\renewcommand{\cD}{{\mathcal D}}

\newcommand{\cI}{{\mathcal I}}
\newcommand{\cJ}{{\mathcal J}}

\newcommand{\cO}{{\mathcal O}}

\renewcommand{\cR}{{\mathcal R}}

\def\<{\langle}
\def\>{\rangle}

\newcommand{\ext}{{\operatorname{ext}}}
\newcommand{\inn}{\operatorname{in}}

\newcommand{\Spec}{\operatorname{Spec}}

\newcommand{\cProj}{{{\mathcal P}roj}}

\newcommand{\Sing}{{\operatorname{Sing}}}

\newcommand{\Gal}{{\operatorname{Gal}}}

\newcommand{\divv}{\operatorname{div}}

\newcommand{\lcm}{{\operatorname{lcm}}}

\newcommand{\red}{{\operatorname{red}}}

\def\:{{\colon}}
\def\.{{,\dots,}}

\newcommand{\double}{\genfrac..{0pt}1
{\raise -1pt\hbox{$\scriptstyle\longrightarrow$}}{\raise 3pt\hbox
{$\scriptstyle\longrightarrow$}}}

\renewcommand{\setminus}{\smallsetminus}

\def\tototi{\mathbin{\mathop{\otimes}\limits^{\raise-1pt\hbox
{$\scriptscriptstyle {\rm L}$}}}}

\def\indlim{\mathop{\vrule width0pt height7pt depth
4pt\smash{\lim\limits_{\raise 1pt\hbox to 14.5pt
{\rightarrowfill}}}}}
\def\projlim{\mathop{\vrule width0pt height7pt depth
4pt\smash{\lim\limits_{\raise 1pt\hbox to 14.5pt
{\leftarrowfill}}}}}

\newcommand\displaceamount{3pt}

\newcommand{\doubledown}{\ar@<\displaceamount>[d]\ar@<-\displaceamount>[d]}

\newcommand{\doubleup}{\ar@<\displaceamount>[u]\ar@<-\displaceamount>[u]}

\newcommand{\doubleright}{\ar@<\displaceamount>[r]\ar@<-\displaceamount>[r]}

\newcommand{\ch}{{\operatorname{char}}}

\newcommand{\ord}{{\operatorname{ord}}}
\newcommand{\rord}{{\operatorname{rord}}}

\newcommand{\gr}{{\operatorname{gr}}}

\newcommand{\inv}{{\operatorname{inv}}}

\newcommand{\inte}{{\operatorname{int}}}

\newcommand{\init}{\operatorname{in}} 

\title{Resolution Except for the Normal-Crossing Locus and Galois actions}
\begin{document}


%

\author[J. W{\l}odarczyk] {Jaros{\l}aw W{\l}odarczyk}
\address{Department of Mathematics, Purdue University\\
150 N. University Street,\\ West Lafayette, IN 47907-2067}
\email{wlodarcz@purdue.edu}

\thanks{This research is supported by  BSF grant 2014365 and Simons Foundation grant  MPS-TSM-00008103} 

\date{\today}

\begin{abstract}
In characteristic zero, we construct a canonical, functorial resolution algorithm by weighted blow-ups that strictly preserves the normal crossings (nc) locus, effectively answering Koll\'ar's Problem~9. Operating in full generality, our approach handles both reduced and non-reduced nc singularities alongside the simple normal crossings (snc) exceptional divisor setup, terminating with a normal crossings Deligne-Mumford stack.

The resolution is governed by two fundamental geometric properties: the openness of the nc locus and the topological rigidity of canonical maximal admissible weighted centers (the Center Theorem). We establish these via a direct Galois-theoretic analysis of splitting forms. By viewing general nc singularities as quotients of \'etale-local snc singularities by finite Galois groups permuting their branches, we reveal the intrinsic necessity of weighted blow-ups and stack structures.

Building on the framework of \cite{ATW-weighted} and its logarithmic refinement in \cite{W22,W23}, this structural mechanism yields a robust resolution algorithm. Applications include a canonical compactification of nc Deligne-Mumford stacks and a functorial nc-preserving resolution of subvarieties and stacks.
\end{abstract}\maketitle
\setcounter{tocdepth}{1}
\setcounter{tocdepth}{2}
\tableofcontents

\section{Introduction}

Hironaka's theorem (1964) shows that every algebraic variety over a field of characteristic zero admits a resolution of singularities, and that the exceptional divisor can be arranged to have simple normal crossings (snc). Subsequent functorial algorithms refined this statement by producing simplified canonical resolutions compatible with smooth morphisms.

However, classical desingularization procedures do not distinguish between genuinely singular points and those that already possess a normal crossings structure. Preserving this structure is geometrically and birationally essential: normal crossings (nc) singularities naturally behave as terminal objects of desingularization theory. They arise as boundary strata in moduli problems (e.g., stable curves, KSBA compactifications), in Hodge theory (semistable reduction and degenerations), and in birational geometry (log pairs $(X,D)$ in the Minimal Model Program). Blowing them up further obscures their intrinsic geometry rather than simplifying it.

The guiding principle of this work is that the normal crossings locus should remain strictly fixed by a canonical desingularization procedure. We therefore seek a resolution algorithm that removes all singularities outside the normal crossings locus while leaving that locus untouched, leading to a notion of \emph{resolution up to normal crossings}. 

The immense topological obstacles to this goal were prominently highlighted in the foundational work of Koll\'ar~\cite{Kollar}. Recognizing that blowing up normal crossings obscures intrinsic structure, he posed Problem~9 in~\cite{Kol08}, asking for the description of the smallest class of singularities that must necessarily be admitted if one requires the birational morphism to be an exact isomorphism over the pre-existing nc locus.

A first version of embedded resolution of divisors except for simple normal crossings (snc) appeared in the work of Szab\'o~\cite{Sza94}. More general and canonical resolutions except for SNC singularities were later established across a series of papers (cf.\ \cite{BM12,BLM12,BDMP14,BBL22}), where the desingularization invariant is refined to avoid blowing up snc points and to isolate minimal residual singularities. In that setting, the splitting of components is already visible Zariski-locally. More generally, resolution except for \emph{locally toric} singularities was obtained in~\cite{W22t} using logarithmic and combinatorial methods.

The problem treated here builds on these developments but confronts a fundamentally different geometric obstruction. For general nc singularities, the splitting of branches into distinct, smooth components may occur only after passing to an \'etale or Galois extension. As Koll\'ar's insights suggest, attempting to strictly preserve general nc singularities within the rigid, classical category of schemes invariably leads to a deadlock. The resulting algorithmic sequences generate limiting configurations---minimal singularities---which cannot be lowered further by classical blow-ups without destroying the nc locus itself, and thus must remain intact as inescapable terminal objects. 

Because the nc condition is intrinsically non-local, bypassing this topological deadlock forces the introduction of stack-theoretic methods and more flexible birational transformations, such as weighted blow-ups. This provides a natural framework: a resolution procedure that preserves the nc locus and separates the fused components may be constructed on Galois covering spaces, and seamlessly descended back to the base space via stack-theoretic operations. 

The purpose of this article is to construct a canonical resolution algorithm that completely removes all singularities outside the nc locus while flawlessly preserving every pre-existing nc structure, effectively answering Koll\'ar's problem.

\subsection{Related Work}

During the final preparation of this manuscript, the preprint \cite{BB26} appeared. Shortly thereafter, the brief note \cite{AT26} and later \cite{BBB26} were posted.

In \cite{BB26} and \cite{BBB26}, a version of partial desingularization is obtained for reduced normal crossings singularities using a formal splitting theorem together with a combinatorial analysis of group-circulant singularities. In \cite{AT26}, an abstract framework is outlined for the reduced case based on axiomatic conditions requiring that the class of singularities have ``open support'' and be ``inv-closed.''

While the approach in \cite{BB26} relies on different, external splitting machinery, the strategy sketched in \cite{AT26} essentially converges toward the same inductive principles of openness and rigidity (the Center Theorem) established in this paper. Furthermore, the underlying proof structure proposed in \cite{AT26} specifically the reduction to the splitting of initial forms via the normal cone structurally mirrors the core arguments detailed in our proof.

However, by natively relying on a direct Galois-theoretic analysis, the present article provides a self-contained, elementary treatment that naturally encompasses the general case of non-reduced NC singularities and the SNC setup for exceptional divisors.

\subsection{The General Resolution Principle and the Center Theorem}

Let $X^{\mathrm{nc}}(\mathcal{I})$ denote the locus of points on a smooth variety $X$ where an ideal $\mathcal{I}$ (reduced or not) has at most normal crossings singularities.

The construction of an NC-preserving algorithm rests on two fundamental geometric ingredients. First, we utilize the weighted blow-up resolution framework of McQuillan and Abramovich--Temkin--W\l odarczyk \cite{ATW-weighted, Marzo-McQuillan}, together with its logarithmic refinement developed in \cite{W22, W23}. The centers are chosen as canonical maximal admissible centers supported on the maximal locus of the logarithmic invariant $\mathrm{inv}_{(X,E)}$ \cite{W22, W23}, which refines the invariant of \cite{ATW-weighted}.

Second and this is the key structural property we prove that such canonical weighted centers are topologically rigid with respect to the nc locus. Unlike classical smooth centers, they cannot intersect $X^{\mathrm{nc}}(\mathcal{I})$ partially.

\begin{theorem}\label{center} (Center Theorem)
A canonical maximal admissible weighted center intersects the nc locus $X^{\mathrm{nc}}(\mathcal{I})$ if and only if it is entirely contained in it.
\end{theorem}

The proof of this rigidity forms the main technical part of the paper. Crucially, the combination of this rigidity and the openness of the normal crossings locus provides a robust, abstract mechanism for partial desingularization. We can encapsulate this in the following structural guarantee:

\medskip
\noindent \textbf{General Resolution Principle.} \emph{Let $\mathcal{C}$ be any class of singularities on a smooth variety $X$ that is open and satisfies the Center Theorem (i.e., any canonical maximal admissible center intersecting the locus of $\mathcal{C}$ is entirely contained within it). Then the functorial desingularization algorithm automatically yields a canonical resolution up to the class $\mathcal{C}$.}
\medskip

It is vital to emphasize that the success of this principle relies entirely on the precise geometric nature of the desingularization invariant. The canonical logarithmic invariant $\mathrm{inv}$ developed in \cite{W22, W23} and generalizing the invariant of \cite{ATW-weighted} is uniquely equipped for this problem, satisfying the Center Theorem for the class of nc singularities adapted to an SNC exceptional divisor, as it is natively compatible with the normal crossings structure.

By applying this principle to the class $\mathcal{C} = X^{\mathrm{nc}}(\mathcal{I})$ using our specific invariant $\mathrm{inv}$, the NC-preserving algorithm proceeds seamlessly:
\begin{itemize}
  \item By the Openness property, the nc locus $X^{\mathrm{nc}}(\mathcal{I})$ is an open subset, making the locus of worse singularities $X \setminus X^{\mathrm{nc}}(\mathcal{I})$ a closed set.
  \item At each step, the algorithm selects the canonical admissible center associated with the maximum of the geometric invariant $\mathrm{inv}$ restricted to the closed complement $X \setminus X^{\mathrm{nc}}(\mathcal{I})$.
  \item By the Center Theorem, this chosen center is strictly disjoint from $X^{\mathrm{nc}}(\mathcal{I})$. Consequently, the blow-up completely avoids and perfectly preserves the pre-existing normal crossings structure.
  \item Because the invariant strictly drops globally, the process must terminate. The algorithm stops when no points lie outside the nc locus, so that $X = X^{\mathrm{nc}}(\mathcal{I})$.
\end{itemize}

\subsection{Galois groups, weights, and stacks}

The obstruction to eliminating nc singularities without modifying the nc locus
is already visible in classical examples, most notably the pinch point
(Whitney umbrella), as explained by Koll\'ar
\cite[Paragraph 8]{Kollar}; see also
\cite[Corollary 3.6.10]{Fujino} and \cite[Example 1.7]{BM12}.
At a nonzero point of the axis, the pinch point has two analytic branches
which are interchanged when one loops around the origin.
Any birational morphism that is an isomorphism over the generic point
of the axis preserves this monodromy.
Hence the singularity cannot be eliminated without blowing up the axis itself.

This example reflects a general feature of nc singularities: their branches split only after passing to a finite Galois extension, and the original nc structure is recovered as the quotient by a finite group acting by permutation of branches.

Thus nc geometry is inherently equivariant geometry.
Smooth blow-ups performed equivariantly on the splitting cover
descend to weighted blow-ups on the base,
with weights encoding the stabilizer and ramification data
of the finite group action.

Since these stabilizers are intrinsic to the nc structure,
the natural ambient category is that of Deligne-Mumford stacks:
the stack quotient records the finite isotropy groups,
while the coarse space corresponds to the weighted transform.
Stacks therefore arise canonically in nc-preserving resolution.

\medskip

To make this structure canonical, we develop a Galois theory
of multivariate splitting forms.

Let
\[
F(x_1,\ldots,x_k)=\prod_i \ell_i^{a_i}\in K[x_1,\ldots,x_k]
\]
be a homogeneous form whose linear factors $\ell_i$
split only after a finite extension.
We associate to $F$ a canonical splitting field $K_F$
with Galois group
\[
G_F=\mathrm{Gal}(K_F/K),
\]
where $K$ is the fraction field of a ring $R$,
and denote by $R_F$ the integral closure of $R$ in $K_F$.
The morphism
\[\mathrm{Spec}\, R_F[x_1,\ldots,x_k]\longrightarrow \mathrm{Spec}\, R[x_1,\ldots,x_k]
\]
encodes the splitting behavior of $F$.
Here $R$ is the coordinate ring of the canonical center,
and $x_1,\ldots,x_k$ generate its defining ideal.

We prove that the locus where $F$ has distinct linear factors
coincides with the \'etale locus of this morphism.
Thus the nc locus is precisely the locus where the associated
Galois cover is unramified.
On the splitting cover the initial form splits and the singularity becomes formally simple normal crossings along the canonical center; equivariant smooth centers resolving the ramification locus descend to canonical weighted centers downstairs.

This Galois-weight-stack correspondence is the structural input
in the proof of the Center Theorem and explains
why canonical weighted centers cannot partially intersect
the nc locus.\section{Main Theorems}

Let $X$ be a smooth irreducible variety, or more generally a smooth irreducible
Deligne--Mumford stack, over a field of characteristic~$0$.  
Let $E$ be an SNC divisor on $X$, and let $\mathcal{I} \subset \mathcal{O}_X$
be a coherent ideal.

Denote by
\[
X^{\mathrm{nc}}, \quad
X^{\mathrm{nc}}_k, \quad
X^{\mathrm{nc}}_{\le k}, \quad
X^{\mathrm{nc}}_{\mathrm{red}}, \quad
X^{\mathrm{nc}}_{k,\mathrm{red}}, \quad
X^{\mathrm{nc}}_{\le k,\mathrm{red}}
\]
the loci of points $p \in X$ such that the closed substack
\[
Z := \Spec(\mathcal{O}_X/\mathcal{I})
\]
has, at $p$, either smooth or normal crossings (respectively reduced normal
crossings) singularities adapted to $E$.

More precisely, these loci correspond to points where $\mathcal{I}$
defines a subscheme that is smooth or has normal crossings (NC),
NC of codimension $k$, NC of codimension at most $k$,
reduced NC, reduced NC of codimension $k$, or reduced NC of
codimension at most $k$, all adapted to $E$.

More generally, one may fix any prescribed collection of codimensions
for the allowed smooth or NC singularities adapted to $E$.

We write $X^{\mathrm{nc}}_*$ to denote any one of the above NC loci.
\begin{theorem}[NC-preserving principalization]
\label{thm:princ-NC}
Let $(X,E)$ be a smooth irreducible variety, or more generally a smooth
irreducible Deligne-Mumford stack, over a field of characteristic~$0$,
equipped with an SNC divisor $E$, and let $\mathcal{I}\subset \mathcal{O}_X$
be a coherent ideal.

There exists a canonical semi-continuous invariant
\[
\inv_{(X,E)}(\mathcal{I}) : X \to \Sigma,
\]
with values in a well-ordered set (see Section~\ref{sec:invariant},
Definition~\ref{cntr}),
and a canonical sequence of weighted blow-ups
\[
   \Pi\colon X_N \longrightarrow \cdots \longrightarrow X_1
   \longrightarrow X_0 = X
\]
with smooth centers disjoint from the normal crossings locus
$X^{\mathrm{nc}}_*$ such that:

\begin{enumerate}
\item Each center coincides with the locus where the invariant
$\inv_{(X,E)}(\mathcal{I})$ attains its maximal value outside
$X^{\mathrm{nc}}_*$.

\item The invariant strictly decreases after each blow-up away from
the NC locus. In particular, the process terminates.

\item The morphism $\Pi$ is an isomorphism over $X^{\mathrm{nc}}_*$.
\item The process terminates when the controlled transform
$\Pi^{c}\mathcal{I}\cdot\mathcal{O}_{X_N}$
has normal crossings adapted to the SNC divisor $E' \cup D$.
\item The total transform satisfies
\[
\Pi^{-1}\mathcal{I}\cdot\mathcal{O}_{X_N}
= \mathcal{I}_D \cdot \mathcal{I}',
\]
where
\begin{itemize}
\item $\mathcal{I}_D = \mathcal{O}_{X_N}\!\left(-\sum m_i D_i\right)$
is the invertible ideal sheaf of the exceptional divisor
$D = \sum D_i$, with multiplicities $m_i$
accumulated during the resolution process;

\item the controlled transform $\Pi^c(\mathcal{I})=\mathcal{I}'$
has normal crossings adapted to the SNC divisor
$E' \cup D$, where $E'$ denotes the strict transform of $E$.
\end{itemize}
\item In particular, if $X^{\mathrm{nc}}_*$ coincides with the locus of codimension-one
normal crossings of $\mathcal{I}$, then the total transform
\[
\Pi^{-1}\mathcal{I}\cdot\mathcal{O}_{X_N}
\]
is a normal crossings divisor.
\end{enumerate}
\end{theorem}

\begin{theorem}[Embedded NC-preserving resolution]
\label{thm:embedded-NC}
Let $(X,E)$ be a smooth irreducible variety, or more generally a smooth
irreducible Deligne-Mumford stack, over a field of characteristic~$0$,
equipped with an SNC divisor $E$, and let
$Y \subset X$ be a closed subscheme (respectively, a closed substack).

There exists a canonical sequence of weighted blow-ups
defined by the invariant $\inv_{(X,E)}(Y)$.
\[
   \Pi: X_N \longrightarrow \cdots \longrightarrow X_1
   \longrightarrow X_0 = X
\]
with maximal admissible centers disjoint from the normal crossings locus
$X_*^{\mathrm{nc}}$, such that:

\begin{enumerate}
\item Each center coincides with the locus where the invariant
$\inv_{(X,E)}(Y)$ attains its maximal value outside $X_*^{\mathrm{nc}}$.

\item The invariant strictly decreases after each blow-up away from
the NC locus. In particular, the process terminates.

\item The morphism $\Pi$ restricts to an isomorphism over
$X_*^{\mathrm{nc}}$.

\item The strict transform $\Pi^s(Y)\subset X_N$ has normal crossings
adapted to the SNC divisor $E' \cup D$, where
$D$ is the exceptional divisor and $E'=\Pi^s(E)$.
\end{enumerate}
\end{theorem}

\begin{theorem}[Non-embedded NC-preserving resolution]
\label{thm:nonembedded-NC}
Let $Y$ be a  finite type scheme or a Deligne-Mumford  stack over a field of characteristic $0$. Let $Y^{\mathrm{nc}}$ (respectively $Y^{\mathrm{nc}}_{\red}$) be the NC locus of  $Y$ containing  all the points with at most NC singularities (resp. reduced NC singularities). There exists a  canonical resolution except NC singularities that is a proper birational morphism of Deligne-Mumford  stacks:\[
   \pi: Y' \to Y
\]
such 
  \begin{enumerate}
  \item 
  $Y'$ is a DM stack with at most NC singularities
   \item $\pi$ is an isomorphism over the NC locus $Y^{\mathrm{nc}}$.  
  \item  The inverse image
\[
\pi^{-1}(\Sing(Y))
\]is an SNC-adapted divisor on $Y'$.
\end{enumerate}
\end{theorem}

\begin{theorem}[NC compactification]\label{NC-c}
Let $X$ be a separated Deligne--Mumford stack over a field of characteristic $0$ 
with at most normal crossing singularities (reduced or non-reduced). 
Then there exists a compactification 
\[
X \subset \overline{X}
\]
such that $\overline{X}$ is a proper Deligne--Mumford stack with at most 
normal crossing singularities containing  $X$ as an open dense substack. 
\end{theorem}

\subsection{Normal Crossing Schemes}
Throughout the paper we work primarily with smooth varieties over a field of characteristic zero, and occasionally with regular schemes. All constructions, definitions, and results extend verbatim, using local charts, to smooth Deligne-Mumford stacks over a field of characteristic zero. Accordingly, we will not distinguish between these settings in what follows.
\begin{definition} Let $X$ be a regular scheme, and $E$ be an SNC divisor on $X$. Then  a system or a partial system $x_1,\ldots,x_k$ of local parameters on $X$ is called {\it adapted} to $E$ if any component of $E$  can be written as a Weil divisor ${\rm div}(x_i)$.  At a point \( p \in U \), a coordinate \( x_i \) is called \emph{divisorial} if \( p \in V(x_i) = E_i \) for some component \( E_i \subseteq E \), and \emph{free} otherwise. 
\end{definition}
\begin{definition}[Normal Crossings Divisor and Ideal]
Let $X$ be a regular scheme and let $E$ be an SNC divisor on $X$.

\smallskip
\noindent
\textup{(1)} An effective Cartier divisor $D$ on $X$ is said to have 
\emph{normal crossings (NC) adapted to $E$} at a point $p \in X$ 
if there exists an \'etale neighborhood $U \to X$ of $p$ and a regular 
system of parameters
\[
x_1, \ldots, x_n \in \mathcal O_{U,p},
\]
adapted to the pullback $E_U$ of $E$, such that the pullback of $D$ to $U$ 
is defined by
\[
x_1^{a_1} x_2^{a_2} \cdots x_r^{a_r} = 0,
\]
for some $1 \le r \le n$ and integers $a_i \ge 1$. 
If, in addition, all $a_i=1$, we say that $D$ is a 
\emph{reduced normal crossings divisor}.

\smallskip
\noindent
\textup{(2)} More generally, a coherent ideal sheaf 
$\mathcal I \subset \mathcal O_X$, or the associated closed subscheme 
$\Spec(\mathcal O_X/\mathcal I)$, is said to have 
\emph{normal crossings (NC)} (respectively \emph{reduced normal crossings}) 
of codimension $k$ adapted to $E$ at $p \in X$ 
if there exists an \'etale neighborhood $U \to X$ of $p$ 
and a regular system of parameters
\[
x_1, \ldots, x_n \in \mathcal O_{U,p},
\]
adapted to $E_U$, such that
\[
\mathcal I\mathcal O_U
=
(x_1, \ldots, x_{k-1}, f),
\qquad
f = x_k^{a_k} x_{k+1}^{a_{k+1}} \cdots x_r^{a_r},
\]
where $x_1, \ldots, x_{k-1}$ are free, and the integers $a_i \ge 1$ (respectively $a_i=1$).

Equivalently, the smooth subscheme $V(x_1,\ldots,x_{k-1})$ of 
codimension $k-1$ carries a normal crossings divisor defined by $f$.

If such presentations exist on an open neighborhood of $p$, 
we say that $D$ (respectively $\mathcal I$) is 
\emph{simple normal crossings (SNC)} adapted to $E$.
\end{definition}
\begin{definition}[NC schemes (non-embedded)] \label{nc2}
A scheme $Y$ is called \emph{normal crossings}
(resp.\ \emph{reduced NC})
if, Zariski-locally on $Y$, there exists a closed embedding
\[
  i : Y \hookrightarrow X
\]
into a regular scheme $X$ such that the defining ideal
$\mathcal{I}_Y \subset \mathcal{O}_X$
is an NC ideal in the sense above.

A Cartier divisor $E$ on an NC scheme $Y$
is called \emph{SNC-adapted}
if, Zariski-locally on $Y$, there exists a closed embedding
\[
  i : (Y,E) \hookrightarrow (X,D)
\]
into a regular scheme $X$ equipped with an SNC divisor $D$
such that:
\begin{itemize}
\item $\mathcal{I}_Y \subset \mathcal{O}_X$ is an NC ideal adapted to $D$, and
\item $i^{-1}(D)=E$ as Cartier divisors on $Y$.
\end{itemize}
\end{definition}\section{Resolution by Weighted blow-ups} 
In this section we provide a brief summary of stack-theoretic and cobordant weighted blow-ups in order to clarify the general context of the nc algorithm, which can be viewed as a modification of the standard weighted blow-up resolution via the Center Theorem. The proofs of the Center Theorem and the main results rely only on very basic properties of weighted blow-ups. Consequently, readers interested primarily in the core arguments may safely skip this section and refer back to it only when needed.
\subsubsection{\texorpdfstring{\textbf{$\QQ$-ideals}}{Q-ideals}}
We start by recalling the basic notions of ideal sheaves and Rees algebras we need, starting by the notion of {\it valuative $\QQ$-ideals $\cJ$} (or simply  {\it $\QQ$-ideals $\cJ$}), introduced in \cite{ATW-weighted} and closely related to \emph{idealistic exponents} of \cite{Hironaka-idealistic}. In the simplest case they could be directly related to the notion of  {\it rational powers of ideals} considered by Huneke-Swanson (\cite[Section 10.5]{HS06}). They naturally generalize the notion of ideals and can be used for the compact description of the centers. The following definition is an adaptation of the more general definition from \cite{ATW-weighted} to a simpler case as in \cite{ W22,W23}:

\begin{definition}[$\QQ$-ideal] By  a {\it $\QQ$-ideal $\cJ$} on a normal noetherian scheme $X$ we mean an equivalence  class  of  formal expressions  $\cJ^{1/n}$, where $\cJ$ is the ideal on $X$ and $n\in \NN$ is a natural number.  We say that $\cJ^{1/n}$ and $\cI^{1/m}$ are {\it equivalent} if the integral closures of
$\cJ^{m}$ and $\cI^{n}$ are equal:
\[
(\cJ^m)^\inte=(\cI^m)^\inte.
\]
\end{definition}

\begin{remark} If $\cJ^{1/n}$ and $\cI^{1/m}$ represent the same $\QQ$-ideal for some ideals $\cI$, and $\cJ$ and $m,n\in \NN$ then the ideals $\cI$ and $\cJ$ are called {\it projectively equivalent} with respect to the coefficient $m/n$ following Rush \cite{Rush07}.

\end{remark}
 One naturally considers  the operations of addition $\cI^{1/n}+\cJ^{1/m}=(\cI^m+\cJ^n)^{1/mn}$, and multiplication $\cI^{1/n}\cdot\cJ^{1/m}=(\cI^m\cJ^n)^{1/mn}$. Any  $\QQ$-ideal $\cJ=\cI^{1/n}$ on $X$ defines the {\it ideal of sections $$\cJ_X:=\{f\in \cO_X\mid f^n\in \cI^{\inte}\}.$$}
 are known as the {\it rational powers of $\cI$} and were considered by Huneke-Swanson in \cite{HS06}. Consequently one can associate  with any $\QQ$-ideal $\cJ$  the graded algebra of $\QQ$-ideals
 
\subsection{ Canonical invariants and maximal admissible centers}
\subsubsection{Canonical invariant in the non-logarithmic setup} In \cite{ATW-weighted}, a \emph{center} is a $\QQ$-ideal  of the form $\cJ=(x_1^{a_1},\ldots,x_k^{a_k})$, where $(x_1,\ldots,x_k)$ is a partial  system of local parameters.  We shall always assume that $$a_1\leq\ldots\leq a_k$$ in the presentation of the center. We now introduce the canonical invariant $\inv$ used in the weighted 
resolution of singularities, following~\cite{W22,W23}, as a variation 
of the construction in~\cite{ATW-weighted}. 
The invariant originates from the inductive definition given in~\cite{ATW-weighted}. 
In~\cite{W22,W23} it was reformulated both in the framework of 
$\QQ$-ideals and in terms of Rees algebras, and further extended to the 
logarithmic setting for the purposes of weighted logarithmic resolution.

We begin with the non-logarithmic setup. To this end, we consider the ordered set
\[
\QQ_+ := \{0,1,2,\ldots\} \cup \{\infty\},
\]
equipped with the natural total order (and later with the induced 
lexicographic order on finite sequences).\begin{definition}[Canonical invariant]\label{def:CanonicalInvGeneral}
Let $\cI \subset \cO_X$ be a coherent ideal and let $p\in X$.
A  center
\[
\cJ=(x_1^{a_1},\ldots,x_k^{a_k})
\]
(in a system of regular parameters at $p$) is called
\emph{$\cI$-admissible at $p$} if
\[
\cI_p \subset \cJ_{X,p}.
\]

The \emph{canonical invariant} of $\cI$ at $p$ is the maximal multi-index
\[
\inv_p(\cI)
:= \max_{\mathrm{lex}}
\left\{
(a_1,\ldots,a_k)
\;\middle|\;
\cI_p \subset (x_1^{a_1},\ldots,x_k^{a_k})_{X,p}
\right\},
\]
where the maximum is taken with respect to the lexicographic order
over all $\cI$-admissible monomial centers at $p$, with $(x_1,\ldots,x_k)$ ranging over all partial regular systems of parameters of $\cO_{X,p}$.

An $\cI$-admissible center $\cJ$ is called
\emph{canonical} (or \emph{maximal admissible})
if its exponent vector coincides with $\inv_p(\cI)$.
\end{definition}
It is proven that canonical center $\cJ$ of any ideal $\cI$ exists and is unique.

\subsubsection{Rational Rees algebra} It is convenient to consider rationally graded Rees algebras in our context.
\begin{definition}\cite{ W22,W23} A \emph{rational Rees algebra} (or simply \emph{Rees algebra}) is a finitely generated $\cO_X$-algebra
\[
R = \bigoplus_{a \in \Gamma} R_a t^a \subset \cO_X[t^{1/w_R}],
\]
where $\Gamma$ is a finitely generated additive subsemigroup of $\QQ_{\geq 0}$, $R_0 = \cO_X$, and $R_a \cdot R_b \subseteq R_{a+b}$. The minimal $w_R \in \QQ_{>0}$ such that $\Gamma \subseteq (1/w_R)\cdot \ZZ_{\geq 0}$ is called the \emph{grading denominator}.
\end{definition}

\subsubsection{Regular centers vs. Rees centers}
There is a natural correspondence between the center $\cJ:=(x_1^{a_1},\ldots,x_k^{a_k})$  and their Rees version:
\begin{itemize}
\item Rees centers $\cA_\cJ=\cO_X[x_1 t^{1/a_1}, \ldots, x_k t^{1/a_k}]^\inte$, with $a_i\in \QQ_{>0}$.
\item Extended Rees centers  $\cA_\cJ^{\ext}=\cO_X[t^{-1/w_A},x_1 t^{1/a_1}, \ldots, x_k t^{1/a_k}]$, where $w_A=\lcm(a_1,\ldots,a_k)$
\end{itemize}

Here $"\inte"$ stands for the internal closure in $\cO_X[t^{1/w_A}]$. One can show that that the nonnegative part 
$$\cA^{\ext}_{\cJ \geq 0}:=\cO_X[t^{-1/w_A},x_1 t^{1/a_1}, \ldots, x_k t^{1/a_k}] \cap \cO_X[t^{1/w_A}]$$ coincides with $$\cA^{\ext}_{\cJ \geq 0}=\cA_\cJ.$$

\bigskip
\subsection{Canonical invariant for Rees algebras}

The invariant can be expressed most naturally in terms of 
Rees algebras with rational gradings.

\begin{definition}[Canonical invariant {\cite[Def.~3.1.7]{ W22,W23}}]
\label{def:CanonicalInvGeneral-Rees}
Let $X$ be a regular scheme, and let $\mathcal{I}$ be an ideal sheaf, or more 
generally let $\mathcal{R}$ be a Rees algebra on $X$.  

A \emph{weighted center} at a point $p\in X$ is an algebra of the form
\[
   \mathcal{A}_J
   = \mathcal{O}_{X,p}
     \bigl[x_1 t^{1/a_1},\ldots,x_k t^{1/a_k}\bigr]^{\mathrm{int}},
\]
where $a_1\le\cdots\le a_k$ are positive rational numbers and 
$(x_1,\ldots,x_k)$ is a regular system of parameters at $p$.

We say that $\mathcal{A}_J$ is \emph{$\mathcal{I}$-admissible} 
(resp.\ \emph{$\mathcal{R}$-admissible}) at $p$ if
\[
   \mathcal{I}_p\, t \subset \mathcal{A}_{J,p},
   \qquad
   \text{resp. } 
   \mathcal{R}_p \subset \mathcal{A}_{J,p}.
\]

The \emph{canonical invariant} at $p$ is defined by
\[
   \inv_p(\mathcal{I})
   \ (\text{resp.\ } \inv_p(\mathcal{R}))
   :=
   \max_{\mathrm{lex}}
   \left\{
     (a_1,\ldots,a_k)
     \,\middle|\,
     \mathcal{A}_J
     \text{ is admissible at } p
   \right\}.
\]

The corresponding center
$
   J=(x_1^{a_1},\ldots,x_k^{a_k})
$
and the associated weighted algebra $\mathcal{A}_J$
are called \emph{maximal admissible centers} if their exponent vector
equals the canonical invariant.
\end{definition}

\begin{remark}
\label{rem:inv-compat-QIDL}
Let 
\[
\mathcal{A}_J
=
\mathcal{O}_{X,p}
\bigl[x_1 t^{1/a_1},\ldots,x_k t^{1/a_k}\bigr]^{\mathrm{int}}
=
\bigoplus_{q\in\QQ_{\ge0}} (\mathcal{A}_J)_q t^q
\]
be a weighted center and let 
\(
J=(x_1^{a_1},\ldots,x_k^{a_k})
\)
be the corresponding $\QQ$-ideal.
Then the degree $1$ component satisfies
$
(\mathcal{A}_J)_1 = J_{X,p}.
$
Hence
\[
\mathcal{I}_p\, t \subset \mathcal{A}_{J,p}
\quad\Longleftrightarrow\quad
\mathcal{I}_p \subset J_{X,p}.
\]
Therefore the definition of $\inv_p(\mathcal{I})$ via weighted centers
coincides with the previous $\QQ$-ideal definition.
\end{remark}

\begin{example} $\cI t=(x^2+y^3+w^4)t\subset  \cO_X[x t^{1/2},yt^{1/3},wt^{1/4}]^{\inte} $
$$\inv_0(\cI)=(2,3,4).$$
\end{example}

\subsection{Weighted, stack-theoretic, and cobordant blow-ups}

For the weighted ideal 
\[
\cJ = (x_1^{1/w_1}, \ldots, x_k^{1/w_k}),
\]
the classical weighted blow-up is given by
\[
Y = \cProj\!\big(\cO_X[x_1 t^{w_1}, \ldots, x_k t^{w_k}]^{\!\inte}\big).
\]

The stack-theoretic weighted blow-up (see \cite[3.1]{ATW-weighted}) is the quotient Deligne-Mumford stack
\[
\left[
\big(\Spec_X(\cO_X[x_1 t^{w_1},\ldots,x_k t^{w_k}]^{\!\inte}) 
   \setminus V(t^{w_1}x_1,\ldots,t^{w_k}x_k)\big)
   /G_m
\right].
\]

The \emph{cobordant weighted blow-up} was introduced and studied in 
\cite{ W22,W23}. 
It is defined by the morphism \(B_+ \to X\), where
\[
B = \Spec_X\!\big(\mathcal{O}_X[t^{-1}, x_1 t^{w_1},\ldots,x_k t^{w_k}] \big),
\qquad
B_+ := B \setminus V(x_1 t^{w_1},\ldots,x_k t^{w_k}).
\]
Its stack-theoretic version is the morphism
\[
[B_+/G_m] \to X,
\]
 which provides a presentation of the stack-theoretic weighted blow-up,
independently considered in~\cite{Quek-Rydh}.

Here $V$ is the vertex typically representing the worst singularities of the transform which are linked to the singularities of the center. The vanishing locus $D=V(t^{-1})$ is  the exceptional divisor mapping to the center and $V(\cJ)$ such that $$ B_-:= B \setminus V(t^{-1})\simeq X\times  G_m$$
is the trivial part of the cobordant blow-up. (\cite{ W22,W23}  )

\begin{figure}[ht]
\centering
\begin{tikzpicture}[scale=1.2, every node/.style={font=\small}, >=Latex]

\begin{scope}[shift={(0,0)}]

\foreach \x in {-1.3, 1.3} {
  \draw[thick] (\x,-0.5) -- (\x,0.65);
  \draw[thick, ->] (\x,0.65) -- (\x,1.15); 
  \draw[thick] (\x,1.15) -- (\x,1.8);
}

\draw[blue, thick, ->] (0,-0.2) -- (0,0.78) node[below left, blue] {};
\node[blue,thick] at (0,0.4) {$V$ vertex};

\foreach \angle in {140, 120, 90, 60, 40}
  \draw[green!50!black, thick, ->] (0,0.8) -- ++(\angle:1.2);

\node[green!50!black] at (1.1,2) {$D$ exc. divisor};

\draw[very thick] (-1.5,-0.8) -- (1.5,-0.8);
\fill[red] (0,-0.8) circle (1.5pt);
\node[below, red] at (0,-0.8) {$V(\mathcal{J})$ center};

\node at (0,-2.0) {$B$ full cobord. blow-up};

\end{scope}

\begin{scope}[shift={(3.5,0)}]

\foreach \x in {-1.3, 1.3} {
  \draw[thick] (\x,-0.5) -- (\x,0.65);
  \draw[thick, ->] (\x,0.65) -- (\x,1.15);
  \draw[thick] (\x,1.15) -- (\x,1.8);
}

\foreach \angle in {140, 120, 90, 60, 40}
  \draw[green!50!black, thick, ->] (0,0.8) -- ++(\angle:1.2);

\node[green!50!black] at (1.1,2) {$D$ exc. divisor};

\node[green!50!black,very thick,draw, circle, inner sep=1.5pt] at (0,0.8) {};
\draw[very thick] (-1.5,-0.8) -- (1.5,-0.8);
\fill[red] (0,-0.8) circle (1.5pt);
\node[below, red] at (0,-0.8) {$V(\mathcal{J})$ center};

\node at (0,-2.0) {$B_+ = B \setminus V$ cob. blow-up};

\end{scope}

\begin{scope}[shift={(7,0)}]

\foreach \x in {-1.3, -0.9, -0.5, 0.5, 0.9, 1.3} {
  \draw[thick] (\x,-0.5) -- (\x,0.65);
  \draw[thick, ->] (\x,0.65) -- (\x,1.15);
  \draw[thick] (\x,1.15) -- (\x,1.8);
}

\draw[blue, thick, ->] (0,-0.2) -- (0,0.8) node[above right, blue] {$V$};

\draw[very thick] (-1.5,-0.8) -- (1.5,-0.8);
\fill[red] (0,-0.8) circle (1.5pt);
\node[below, red] at (0,-0.8) {$V(\mathcal{J})$ center};

\node at (0,-2.0) {$B \setminus D = X \times \mathbb{G}_m$};

\end{scope}

\end{tikzpicture}
\caption{Cobordant blow-up: the role of the vertex $V$ and the exceptional divisor $D$. (Reproduced from \cite{ W22}.)}\label{F1}
\label{fig:cobordant-blowup}
\end{figure}
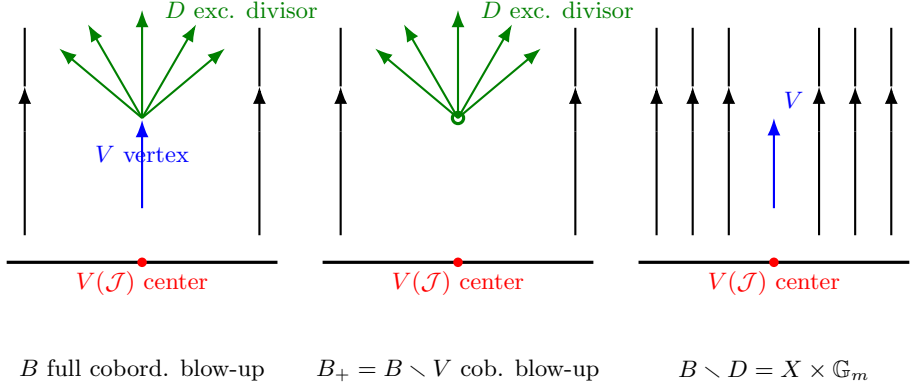

\begin{remark}
The \emph{cobordant blow-up} from   \cite{ W22,W23}   is the natural extension of the stack-theoretic weighted blow-up as in \cite{ATW-weighted} and is obtained by directly generalizing the classical extended Rees algebra for the ideal of a smooth center, (see \cite{Rees56,Eis95,HS06})
\[
\cO_X[t^{-1},\, t\cI],
\]
to the weighted setting, providing a smooth birational model for weighted blow-ups.  
Its key features are:

\begin{enumerate}
\item \textbf{Birational-cobordism interpretation:}  
   The morphism \(B\to X\times \bA^1\) is a smooth birational cobordism in the sense used in the proof of the Weak Factorization Theorem \cite{W00} \cite{W03} and presents the weighted blow-up within a unified birational framework.  
   In fact, the induced morphism \(B_+/G_m \to X\) is the ordinary weighted blow-up.

\item \textbf{Smoothness of \(B\):}  
   The total space \(B\) is nonsingular, which yields a simple description of the transforms of ideals, vector fields, and related structures directly on $B$.  
   Since the morphism \(B\to X\) is affine, one works in a single chart with coordinates \(t^{-1}, x_1 t^{w_1},\ldots,x_k t^{w_k}\), which is advantageous even compared to the standard smooth blow-up.  
   The construction was designed for applications beyond the classical case, including positive characteristic and foliations.

\item \textbf{Relation with Cox rings:}  
   The cobordant blow-up admits a natural formulation via Cox rings of the underlying birational morphism, extending the toric picture to more general settings.\cite{W23b}

\item \textbf{Comparison with deformation to the normal cone:}  
   The secondary projection \(B \to \AA^1\) induces a deformation to the normal bundle, giving a certain resemblance to the Fulton--MacPherson deformation  
   \(\mathrm{Bl}_J(X \times \PP^1)\to\PP^1\), where \(\cJ\) is the ideal of the point \(p\in X\times\{\infty\}\), particularly in the classical case where all weights are equal to \(1\). (\cite{Fulton}) 
   The purposes, however, are fundamentally different: the cobordant blow-up is \emph{birational} in nature and provides a smooth presentation of the weighted blow-up, whereas the Fulton--MacPherson construction is \emph{deformational}, designed for normal-cone deformations and not suited to producing smooth weighted models.  
   Indeed, for weighted centers the space \(\mathrm{Bl}_J(X \times \PP^1)\) is typically singular, in contrast with the smooth total space \(B\). On the other hand the local chart of this construction produces the extended Rees algebra as mentioned above which is the language suitable for weighted generalization.
\end{enumerate}


\end{remark}

\subsection{Cobordant blow-ups at the maximal admissible centers}
Let 
\[
\cJ=(x_1^{a_1},\ldots,x_k^{a_k})
\]
be a maximal admissible center.
The admissibility condition can be expressed in terms of the extended
Rees algebra
\[
\cI t \subset 
\cA_\cJ^{\ext}
=
\cO_X\!\left[
  t^{-1/w},\, x_1 t^{1/a_1}, \ldots, x_k t^{1/a_k}
\right].
\]

Rescaling the parameter $t \mapsto t^{w}$, where
$w$ is a multiple of
\[
w_A=\lcm(a_1,\ldots,a_k),
\]
one obtains the algebra of the full cobordant blow-up along $\cJ$:
\[
\cO_B
=
\cO_X\!\left[
  t^{-1},\, x_1 t^{w_1}, \ldots, x_k t^{w_k}
\right],
\qquad
w_i=\frac{w}{a_i}.
\]
In this algebra we have
\[
\cI \cdot t^{w} \subset \cO_B .
\]

The total transform of $\cI$ in $\cO_B$ is 
\[
\cO_B \cdot \cI .
\]
Factoring out the exceptional divisor defined by $t^{-1}$, 
equivalently dividing by the corresponding power of $t^{-1}$,
the ideal
\[
\cO_B \cdot (\cI t^{w}) \subset \cO_B
\]
is the \emph{controlled transform} of $\cI$.\subsection{Maximal contact and coefficient ideal}
Denote by $\mathcal{D}_X$, and more generally by $\mathcal{D}_X^i$ for $i \ge 0$, 
the sheaf of derivations (respectively, differential operators of order $\le i$) 
on a smooth variety $X$. 

For an ideal sheaf $\mathcal{I}$ and an integer $i \ge 0$, 
let $\mathcal{D}_X^i(\mathcal{I})$ be the ideal generated by 
$D(f)$, where $f \in \mathcal{I}$ and $D \in \mathcal{D}_X^i$.

\medskip
\noindent
\textbf{Maximal admissibility relation.}
Consider
\[
\cI t^a \subset \cO_X[m_p t],
\qquad
a=\ord_p(\cI).
\]
Acting by $t^{-1}\cD_X = \mathrm{span}(t^{-1}\partial_{x_i})$ 
enlarges the left-hand side while preserving the right-hand side
\[
\cO_X[m_p t]
=
\cO_X \oplus m_p t \oplus m_p^2 t^2 \oplus \cdots .
\]
In degree $1$ this gives
\[
\cD_X^{a-1}(\cI)t
\subset
m_p t
=
(x_1,\ldots,x_n)t.
\]
Hence there exists a coordinate
\[
x_1 \in \cD_X^{a-1}(\cI),
\]
called a \emph{single Hironaka maximal contact}.

\medskip
\noindent
\textbf{Maximal contact for weighted admissibility.}
Consider now the maximal admissibility relation
\[
\cI t
\subset
\cA_\cJ
=
\cO_{X,p}
[x_1 t^{1/a_1},\ldots,x_k t^{1/a_k}]^{\mathrm{int}} .
\]
By maximality,
\[
a_1=\ord_p(\cI)\le a_2\le\cdots\le a_k.
\]
Acting by the rescaled derivations $t^{-1/a_1}\cD_X$ 
again enlarges the left-hand side while preserving the right-hand side. 
As above, we obtain a distinguished graded coordinate
\[
x_1 t^{1/a_1}
\in
T(\cI)
:=
\cD_X^{a_1-1}(\cI)\, t^{1/a_1}
\subset
\cA_\cJ,
\]
where $T(\cI)$ is called the \emph{cotangent ideal} of $\cI$.

Thus (after a coordinate change) we construct the first graded coordinate 
$x_1 t^{1/a_1}$ of the maximal admissible center.

\medskip
More generally, in the absence of divisors, one obtains a {\it maximal contact} in the form
\[
\overline{x_1}t^{1/a_1}
:=
(x_{11}t^{1/a_1},\ldots,x_{1k_1}t^{1/a_1}),
\]
where $(x_{11},\ldots,x_{1k_1})$ is a maximal partial system of local parameters contained in the cotangent ideal
\[
T(\cI)=\cD_X^{a_1-1}(\cI).
\]\subsection{Coefficient ideals}\cite{ W22}

Given any Rees algebra $R=\bigoplus R_at^a=\cO_X[f_jt^{b_j}]$
the  {\it coefficient ideal} $C_{\overline{x}t^{1/a_1}}(R)$  of  $R$ with respect to $\overline{x}_1t^{1/a}$ is   a Rees $\cO_{V(\overline{x}_1)}$- algebra generated  by  the graded coefficients  of the generators write $f_j$ in $\widehat{\cO}_{X,{V(\overline{x}_1)}}$  as $$f_j\equiv \sum c_{j\alpha} \overline{x}^\alpha =\sum_{  |\alpha|<  b_j a_1} c_{j\alpha} \overline{x}^\alpha + \sum_{  |\alpha|\geq  b_j a_1} c_{j\alpha} \overline{x}^\alpha \in \widehat{\cO}_{X,V(\overline{x}_1)}$$  

\[
C_{\overline{x} t^{1/a_1}}(R) := \mathcal{O}_{V(\overline{x}_1)}\left[c_{j\alpha|{V(\overline{x}_1)}} \, t^{b_j - |\alpha|/a_1} \,\big|\, |\alpha| < b_j a_1 \right],
\]
for the coordinate system $(\overline{x}_1,y)$.
\begin{proposition}{Reduction Property of the Coefficient Ideal} \cite{ W22}
\begin{itemize}
    \item Let $R$ be the Rees algebra of order $a_1$ at a point $p \in X$, with maximal contact given by $\overline{x}_1 t^{1/a_1}$.
    \item Let $\mathcal{A} = \mathcal{O}_X$ be the maximal center for $R$.
\end{itemize}

Then:
\begin{itemize}
    \item The restriction $\mathcal{A}_{|V(\overline{x}_1)}$ is a maximal admissible center for the coefficient ideal $C_{\overline{x}_1 t^{1/a_1}}(R)$.
    \item Moreover, we can write:
    \[
    \mathcal{A} = \mathcal{O}_X\left[\overline{x}_1 t^{1/a_1}, \mathcal{A}_{|V(\overline{x}_1)}\right]^{\mathrm{int}}
    \]
\end{itemize}
\end{proposition}
Consequently we have a very simple inductive formula
\begin{proposition}\cite{ W22}
$$\inv_p(\cR)=(\ord_p(\cR),\inv_p(C_{x_1t^{1/a_1}}(R)),$$
where $a_1=\ord_p(\cR),$ and $x_1t^{1/a_1}$ is a single maximal contact.
\end{proposition}

\subsection{Logarithmic resolution} \label{sec:invariant}

Before passing to the logarithmic setting, recall that the invariant 
$\inv_p(\mathcal{I})$  was originally in\cite{ATW-weighted}
considered for regular schemes without boundary divisors in a slightly different language.  
In the presence of a simple normal crossings divisor
\[
   D=\bigcup_i D_i,
\]
this invariant admits a natural logarithmic refinement  developed in ~\cite{ W22,W23}, 
which is 
compatible with the previously considered  construction used in the boundary-free case.

Choose local coordinates $(x_1,\ldots,x_n)$ {\it adapted to $D$}, such that each component $D_i$ 
is locally given by $x_j=0$ for some $j$.  
We distinguish:
\begin{itemize}
  \item \emph{free} coordinates $x_i$ (not vanishing along $D$), and
  \item \emph{divisorial} coordinates $x_j$ defining components of $D$.
\end{itemize}

To reflect that divisorial variables behave as heavier components, associate 
to the weights $a_i$ the symbols
\[
   b_j =
   \begin{cases}
     a_j, & \text{if } x_j \text{ is free},\\[4pt]
     a_{j+}, & \text{if } x_j \text{ is divisorial},
   \end{cases}
\]
where $a_{j+}$ is a formal symbol satisfying
\[
a_j < a_{j+} < b
\quad \text{for every } b>a_j \text{ in } \mathbb{R}.
\]
\begin{definition}[Canonical invariant {\cite[Def.~3.1.7]{ W22,W23}}]
\label{cntr}
Let $X$ be a regular scheme, with an SNC divisor $E$ and let $\mathcal{I}$ be an ideal sheaf, or more 
generally let $\mathcal{R}$ be a Rees algebra on $X$.  

A \emph{weighted center} at a point $p\in X$ is a Rees algebra of the form
\[
   \mathcal{A}_J
   = \mathcal{O}_{X,p}
     \bigl[x_1 t^{1/a_1},\ldots,x_k t^{1/a_k}\bigr]^{\mathrm{int}},
\]
where $a_1\le\cdots\le a_k$ are positive rational numbers and 
$(x_1,\ldots,x_k)$ is a regular system of parameters at $p$ adapted to $E$.

We say that $\mathcal{A}_J$ is \emph{$\mathcal{I}$-admissible} 
(resp.\ \emph{$\mathcal{R}$-admissible}) at $p$ if
\[
   \mathcal{I}_p\, t \subset \mathcal{A}_{J,p},
   \qquad
   \text{resp. } 
   \mathcal{R}_p \subset \mathcal{A}_{J,p}.
\]

The \emph{canonical logarithmic invariant} at $p$ is defined by
\[
   \inv_p(\mathcal{I})
   \ (\text{resp.\ } \inv_p(\mathcal{R}))
   :=
   \max_{\mathrm{lex}}
   \left\{
     (b_1,\ldots,b_k)
     \,\middle|\,
     \mathcal{A}_J
     \text{ is admissible at } p
   \right\}.
\]

The corresponding center
$
   J=(x_1^{a_1},\ldots,x_k^{a_k})
$
and the associated weighted algebra $\mathcal{A}_J$
are called \emph{maximal admissible centers} if their exponent vector
equals the canonical invariant.
\end{definition}

Centers and the invariant can be presented more compactly as:
\[ \cA_\cJ = \cO_X[\overline{x}_1t^{1/a_1}, \ldots, \overline{x}_kt^{1/a_k}]^\inte, \]
with blocks \( \overline{x}_i = (x_{i1}, \ldots, x_{ik_i}) \), where $a_1<\ldots<a_k$  and associated weights \( \overline{b}_i :=(\underbrace{a_i,\ldots,a_i}_{\text{free}}, \underbrace{a_{i+},\ldots,a_{i+}}_{\text{divisorial}})
 \).
Consequently for an ideal $\cI$ we have
\[ \inv_p(\cI) = \max \{ (\overline{b}_1, \ldots, \overline{b}_k) \mid \cI t \subseteq \cA_\cJ \}. \]
For Rees algebra $R$, we repectively write the invariant:
\[ \inv_p(R) := \max \{ (\overline{b}_1, \ldots, \overline{b}_k) \mid R \subseteq \cA_\cJ \}. \]

The centers $\cA_\cJ$ (or respectively $\cJ$ in the $\QQ$-ideal form) for which the maximum is unique is called {\it maximal admissible center} for an ideal $\cI$ or a Rees algebra $R$. These centers are used in the weighted logarithmic resolution as well as in our partial resolution. 
Here the part $\overline{x}_1t^{1/a_1}$ of a maximal admissible center is a maximal contact. It can be computed directly from Rees algebra $R$ or ideal $\cI$.

For a Rees center $\mathcal A_{\mathcal J}$, define
\[
\inv(\mathcal J)
=\inv(\mathcal A_{\mathcal J})
=(\overline b_1,\ldots,\overline b_k),
\qquad
\inv^1(\mathcal A_{\mathcal J})=\overline b_1,
\qquad
\ord(\mathcal A_{\mathcal J})=a_1.
\]

Note that $\mathcal J$ is maximal admissible for $\mathcal I$
(resp.\ for $R$) at $p\in X$ if and only if
\[
\inv(\mathcal J)=\inv_p(\mathcal I)
\quad
(\text{resp.\ }\inv(\mathcal J)=\inv_p(R)).
\]

Define
\[
\inv^1_p(R)
:= \max\bigl\{\inv^1(\mathcal \cJ)\mid R_p\subseteq \mathcal A_{\cJ,p}\bigr\}.
\]
This provides a useful first approximation to the invariant $\inv$.
It extends the notion of the order of the Rees algebra $R$ at $p$,
defined by
\[
\ord_p(R)
:= \min_{a{>0}}
\left\{\frac{\ord_p(R_a)}{a}\right\}.
\]

In particular,
\[
\ord_p(R)=\ord(\mathcal A_{\mathcal J})=a_1,
\]
where $\mathcal {\mathcal J}
$
is a maximal admissible center for $R$.
Let \( R = \bigoplus R_b \) be a Rees algebra of order \( a_1 \) at a point \( p \in X \), so that \( R_b \subseteq m_p^{{\lceil b a_1 \rceil}} \) for all \( b \). Then the associated graded algebra of \( R \) at \( p \) \[
\gr_p(R) := \bigoplus_{ba_1\in \NN} \frac{R_b + m_p^{{ b a_1 } + 1}}{m_p^{{ b a_1 } + 1}} t^b 
\subseteq\kappa(p)\left[\left( \frac{m_p}{m_p^2} \right) t^{1/a_1} \right],
\]
where $\kappa(p)$ is the residue field at $p\in X$.
This generalizes the  standard construction of $\gr_p(\cI)$.
We shall use in this paper \begin{lemma}\cite[Lemma 2.4.27]{ W22}. \label{gr}If $R$ has order $a_1$ at $p$, then:
\begin{itemize}
  \item $\ord_p(R) = \ord_p(\gr_p(R))$,
  \item $\inv_p^1(R) = \inv_p(\gr_p(R)) = \inv_p^1(\gr_p(R))$.
\end{itemize}
\end{lemma}

Moreover we can write in the logarithmic setting
\begin{proposition}
$$\inv_p(\cR)=(\inv^1_p(\cR),\inv_p(C_{\overline{x}_1t^{1/a_1}}(R)),$$
where $a_1=\ord_p(\cR),$ and $\overline{x}_1t^{1/a_1}$ is a  maximal contact of $R$.
\end{proposition}
\subsection{Resolution by weighted centers}

The main theorem of~\cite{ATW-weighted} establishes a canonical version of
resolution by weighted blow-ups.
For our purposes, we require its canonical \emph{logarithmic} counterpart,
which was presented in~\cite{ W22,W23}.
Recall that the approach of~\cite{ATW-weighted} does not treat the logarithmic case.

The logarithmic version was developed in full generality in~\cite{ W22,W23},
by extending and refining the formalism originally introduced in the
boundary-free setting.
A later and more elaborate treatment of the logarithmic situation was given
in~\cite{ABQTW}, following a somewhat different approach, which however is computationally more involved and not well suited for the methods and the results developed in this paper.

\begin{theorem}[Canonical weighted logarithmic resolution 
{\cite{W22,W23}}]
\label{thm:weighted-log-resolution}
(See also~\cite{ATW-weighted,Marzo-McQuillan,ABQTW,Quek}.)

Let $(X,E)$ be a smooth logarithmic pair over a field of characteristic~$0$,
and let $\mathcal{I}$ be an ideal (or Rees algebra) on $X$.
There exists a semi-continuous invariant
\[
\inv_{(X,E)}(\mathcal{I}) : X \to \Sigma
\]
with values in a well-ordered set, and an associated maximal admissible center
$\mathcal{J}$, such that:

\begin{enumerate}
\item $V(\mathcal{J})$ is precisely the locus where $\inv$ attains its maximal value;

\item if $\sigma : X_1 \to X$ is the weighted blow-up of $\mathcal{J}$,
with controlled transform $\mathcal{I}_1$ and divisor $E_1$, then
\[
\max \inv_{(X_1,E_1)}(\mathcal{I}_1)
<
\max \inv_{(X,E)}(\mathcal{I});
\]

\item the construction is functorial for smooth morphisms.
\end{enumerate}

Iterating the procedure yields a canonical sequence of weighted blow-ups
at maximal admissible centers adapted to $E$.

If $\mathcal{I}=\mathcal{I}_Z$ is the ideal of a closed (generically reduced) subscheme
(or Deligne-Mumford substack) $Z\subset X$, the strict transform of $Z$
becomes smooth and meets the exceptional divisor with SNC,
giving a canonical embedded logarithmic resolution of $Z$.

In general, the process yields principalization of $\mathcal{I}$:
its total transform becomes a locally monomial ideal supported on an SNC divisor.
\end{theorem}\subsection{Invariant associated with NC singularities}
\begin{lemma} \label{invar} Let $\cI$ be  a Normal Crossing of codimension $r+1$ at $p\in X$ so that in an \'etale neighborhood $\cI$ admits a presentation
\[
\cI = (u_1, \ldots, u_r, m),
\qquad 
m = u_{r+1}^{a_{r+1}} \cdots u_k^{a_k},
\]
where $u_1, \ldots, u_k$ form a partial system of local parameters and 
$a_{r+1} + \cdots + a_k = d$, where $u_1,\ldots,u_{r+t}$ are free and  $u_{r+t+1},\ldots,u_{r}$, are divisorial. 
Then the maximal admissile center at $p\in X$ is given by $$\cJ=(u_1,\ldots,u_{r},u^d_{r+1},\ldots,u^d_{r+t},u^d_{r+t+1},\ldots,u^d_{k}),$$ and the invariant is given by $$\inv_p(\cI)=\inv(\cJ)=(\underbrace{1, \ldots, 1}_{r \text{ times}}, \underbrace{d, \ldots, d}_{t \text{ times}},\underbrace{d_+, \ldots, d_+}_{k-r-t \text{ times}})$$
\end{lemma}
\begin{proof} We follow the algorithm from \cite[Section 3.3]{ W22}. We compute the invariant in the given \'etale neighborhood. Consider the algebra $R_1=\cO_X[\cI t]$.
Its order is equal $$\ord_p(R_1)=\ord_p(\cI)=1,$$ with the tangent ideal $$T(\cI)=\cD^{1-1}(\cI)=\cD^0(\cI)=\cI,$$ and a graded multiple maximal contact $\overline{u}_1t$ where $$\overline{u}_1:=(u_1,\ldots,u_{r}).$$
Then $$R_2:=C_{\overline{u}_1t}(R_1)=(mt)$$ is generated by the only coefficient $m$ with respect to presenataion in $\overline{u}_1$. Thus $$\ord_p(R_2)=\ord_p(mt)=\ord_p(m)=d,$$ In this case multiple maximal contact is given by 
$\overline{u}_2t^{1/d}$, where $$\overline{u}_2:=(u_{r+1},\ldots,u_k)\subset T(m)=\cD^{d-1}(m).$$ Note that
$R_3=C_{\overline{u}_2t^{1/d}}(R_2)=0$. This determines the Rees algebra of the center $$\cA_{\cJ}=\cO_X[\overline{u}_1t,\overline{u}_2t^{1/d}],$$ associated with $\QQ$-ideal center $$\cJ=(\overline{u}_1,\overline{u}_2^d)=(u_1,\ldots,u_r,u^d_{r+1},\ldots,u^d_k)$$ and the invariant $$\inv_p(\cI)=\inv(\cJ)=(\underbrace{1, \ldots, 1}_{r \text{ times}},, \underbrace{d, \ldots, d}_{t \text{ times}},\underbrace{d_+, \ldots, d_+}_{k-r-t \text{ times}}).$$

\end{proof}

\subsection{ Admissible centers and  NC singularities}
\subsubsection{Split neighborhood}
\begin{lemma}[Standard Splitting Lemma]\label{split}
Let $Y \subset X$ be a smooth closed subvariety of a smooth variety $X$.
Then for any point $p \in Y$, there exists an \'etale neighborhood $U$ of $p$ and a smooth morphism
\[
  \pi: U \to Y \cap U
\]
such that $\pi$ restricts to the identity on $Y \cap U$. In other words, $U$ admits a smooth retraction onto $Y \cap U$.
\end{lemma}

\begin{proof}
Choose a system of local parameters (coordinates) $x_1, \dots, x_{n}$ on $X$ near $p$,  such that $Y=V(x_{k+1},\ldots, x_n)$. 
This defines a smooth morphism
\[
  \phi: U_0 \longrightarrow \AA^k, \quad \phi(q) = (x_1(q), \dots, x_k(q)).
\]
whose restriction
 $$\phi_{|Y_0}: Y_0 := Y \cap U_0\to \AA^k$$ is \'etale.  
Define $U$ to be a component of the fiber product $U_0 \times_{\AA^k} Y_0$ which is \'etale over $U_0$ containing the component $Y_0\hookrightarrow Y_0\times_{\AA^k} Y_0$.  
Then $U$ is an \'etale neighborhood of $p\in U_0$ admitting the desired smooth retraction  $U\to Y_0=Y\cap U\hookrightarrow U$.
\end{proof}

\begin{definition}[Split Neighborhood]
Let $Y \subset X$ be a smooth closed subvariety.  
An \'etale neighborhood $U$ of a point $p \in X$ is called a \emph{$Y$-split neighborhood} if there exists a smooth morphism
\[
  \pi: U \longrightarrow Y \cap U
\]
that restricts to the identity on $Y \cap U$.  

If $Y = V(\cJ)$ is defined by a center $\cJ$, we also refer to such a neighborhood as a \emph{$\cJ$-split neighborhood}.
\end{definition}

 \subsubsection{Normal Crossings and the Maximal Admissible Center}


\begin{lemma}\label{etale2}
Let $\cJ$ be a maximal admissible center for an ideal $\cI$ on a smooth
affine variety $X$, and let $p\in V(\cJ)$.

Assume that $\cI$ has normal crossings of codimension $r+1$ and order $d$
in a neighborhood of $p$. Then, after choosing suitable local parameters
$y_1,\ldots,y_k$ at $p$, we may write
\[
\cJ = (y_1,\ldots,y_r,y_{r+1}^d,\ldots,y_{r+t}^d),
\]
with $(y_1,\ldots,y_r)\subset \cI$.

Let $Z:=V(y_1,\ldots,y_r)$. Then $\cI_{|Z}$ has normal crossings of
codimension $1$ at $p$, and
\[
\cJ_{|Z}=(y_{r+1|Z}^d,\ldots,y_{r+t|Z}^d)
\]
is its maximal admissible center. Moreover,
\[
\inv_p(\cI)
=
\inv(\cJ)
=
(\underbrace{1,\ldots,1}_{r\text{ times}},\inv_p(\cI_{|Z})).
\]

Conversely, if $\cI\supset (y_1,\ldots,y_r)=\cI_Z$ and
$\cI_{|Z}$ has normal crossings of codimension $1$ at $p$,
then $\cI$ has normal crossings of codimension $r+1$ at $p$.
\end{lemma}
\begin{proof}If $\cI$ has normal crossings of codimension $r+1$ at $p\in X$,
then there exists an \'etale neighborhood $U\to X$ of $p$ with a presentation
\[
\cI = (y_1,\ldots,y_r,m),
\qquad
m = y_{r+1}^{a_{r+1}}\cdots y_k^{a_k},
\]
where $y_1,\ldots,y_k$ form a partial system of regular parameters
adapted to $E$, and $a_{r+1}+\cdots+a_k=d$.
After multiplying by a unit, we may assume that
$y_1,\ldots,y_r \in \cO_{X,p}$.

Let $Z:=V(y_1,\ldots,y_r)$. Then $\cI_{|Z}$ is NC of codimension $1$
at $p$. Moreover, $(y_1,\ldots,y_r)\subset \cI$ defines a maximal contact,
and
\[
\inv_p(\cI)
=
(\inv^1_p(\cI),\inv_p(\cI_{|Z}))
=
(\underbrace{1,\ldots,1}_{r\text{ times}},\inv_p(\cI_{|Z})).
\]

Conversely, suppose $\cI_Z=(y_1,\ldots,y_r)\subset \cI$ and
$\cI_{|Z}$ is NC of codimension $1$ at $p$.
Passing to an \'etale neighborhood, we may assume that
$Z\subset X$ splits, i.e.\ there exists a smooth projection
$X\to Z$.
Since $\cI_{|Z}$ is NC, there exists an \'etale neighborhood
$U_Z\to Z$ such that $\cO_{U_Z}\cdot\cI_{|Z}$ is generated by a monomial.
Then, in the induced \'etale neighborhood
$X_U := X\times_Z U_Z \to X$,
the ideal $\cO_{X_U}\cdot\cI$ is SNC.\end{proof}



\section{The Galois group of a splitting form}
\subsection{Splitting forms}
\begin{definition}
Let $K$ be a field.
A homogeneous polynomial (or \emph{form})
\[
F = \sum_{\alpha} c_\alpha x^\alpha \;\in\; K[x_1,\ldots,x_k]
\]
of degree $d$ is said to \emph{split} if there exists a field extension $L/K$ such that
\[
   F \;=\; \prod_{i=1}^{r} \ell_i(x_1,\ldots,x_k)^{a_i},
\]
where each $\ell_i(x)$ is a distinct linear form of the shape
\[
   \ell_i(x) = \sum_{j=1}^{k} \alpha_{ij} x_j,
   \qquad \alpha_{ij} \in L,
\]
and $a_i \geq 1$.  
In this situation, we call $F$ a \emph{splitting form}.

We say that $F$ is \emph{monic} if the coefficient of $x_1^d$ in $F$ equals $1$.  
A splitting is called \emph{monic} if in each $\ell_i(x)$ the coefficient $\alpha_{i1}$ of $x_1$ is equal to $1$.
\end{definition}
Recall
\begin{lemma} \label{1} Let $K$ be an infinite field, and 
\[
F = \sum_{\alpha} c_\alpha x^\alpha \;\in\; K[x_1,\ldots,x_k]
\]
be a spitting form of degree $d$. Then by a linear coordinate change $$x'_1=x_1, \quad,x_i'=x_i-\lambda_i x_1,$$ for $i\geq 2, \lambda_i\in K$ and rescaling $F$ one can assume that $F$ is monic. Moreover if $\ch(K)=0$ we can assume that $\lambda_i\in \QQ$.

\end{lemma}
\begin{proof}We use the standard argument from the proof of the Noether
Normalization Theorem (see \cite[Chapter~5, Ex.~16]{AM}).
Consider a general linear coordinate change
\[
x_i' = x_i - \lambda_i x_1 \quad (i=2,\ldots,k),
\]
such that
\[
\alpha := F(\lambda_1,\ldots,\lambda_{k-1},1) \neq 0.
\]
Replacing $F$ by $\alpha^{-1}F$, we may assume that $F$ becomes monic
in $x_1$.\end{proof}
\subsection{Galois extensions of  splitting forms}
\begin{proposition} \label{unique}
With the above notation, suppose $F$ is a  splitting form over an infinite field.   Then by Lemma \ref{1} we can modify it and assume that it is monic.
Then one can associate to monic $F$ a normal field extension
\[
   K_F := K(\alpha_{ij}) \subseteq L,
\]
generated by the coefficients $\alpha_{ij}$ of the linear factors $\ell_i$ appearing in a monic splitting of $F$.  
This field $K_F$ is well-defined up to isomorphism, and is independent of the choice of homogeneous coordinates on $\mathbb{P}^{k-1}$.

The natural action of the Galois group $G_F = \Gal(K_F/K)$ on $K_F$ extends to an action on the polynomial ring
\[
   K_F[x_1,\ldots,x_k],
\]
where $G_F$ acts trivially on the variables $x_j$.  
Under this action, $G_F$ permutes the linear forms $\ell_i$ occurring in the factorization of $F$ over $L$.
\end{proposition}

\begin{definition}
Given a splitting form $F$ over an infinite field, its associated splitting field $K_F$, and the group $G_F=\Gal(K_F/K)$ as above,  
we call $G_F$ the \emph{Galois group of $F$}.\end{definition}

\begin{proof}
We first show the uniqueness of the associated Galois extension.  
Fix a coordinate system $x_1,\ldots,x_k$.  
If $x'_1,\ldots,x'_k$ is another linear coordinate system over $K$, then the coefficients $\beta_{ij}$ of the linear forms with respect to $x'_1,\ldots,x'_k$ are related to the coefficients $\alpha_{ij}$ by an invertible matrix $A=(a_{ij})$ with entries in $K$, namely
\[
   (x'_1,\ldots,x'_k)^T = A (x_1,\ldots,x_k)^T,
   \qquad [\beta_{ij}] = A [\alpha_{ij}].
\]
Hence $K(\alpha_{ij}) = K(\beta_{ij})$.  
Although this change of coordinates may produce linear forms that are not monic, after rescaling each form to make it monic we obtain coefficients $\alpha'_{ij}$ with
\[
   K(\alpha'_{ij}) \subseteq K(\alpha_{ij}).
\]
By symmetry, the reverse inclusion also holds, so $K(\alpha_{ij})=K(\alpha'_{ij})$.  
Thus the splitting field $K_F=K(\alpha_{ij})$ is independent of the choice of homogeneous coordinates.

\medskip

Next, write the splitting in monic form:
\[
   \ell_i(x) \;=\; x_1 + \sum_{i=2}^k a_{ij} x_j,
   \qquad
   F \;=\; \prod^s_{i=1} \ell_i(x)^{a_i}.
\]
Expanding $F$ as a polynomial in $x_1$ with coefficients in $K[x_2,\ldots,x_k]$, we obtain
\[
   F(x_1,\ldots,x_k) 
   \;=\; x_1^m + b_{m-1}(x_2,\ldots,x_k)\,x_1^{m-1} + \cdots + b_0(x_2,\ldots,x_k).
\]

For each $j=2,\ldots,k$, define a univariate polynomial
\[
   \Phi_j(x_1) := F\!\bigl(x_1;\; x_j=-1, \; x_{j'}=0 \ \text{for}\ j'\neq 1,j\bigr)\in K[x_1].\quad \text{Then}
\]
\[
  \Phi_j(x_1) = \prod_{t} (x_1 - \alpha_{ij})^{a_i} \in K[x_1],
\]
so the elements $\alpha_{tj}$ are roots of $\Phi_j(x_1)$.  
Thus $K_F$ is the splitting field of the family of polynomials $\{\Phi_j(x_1)\}_{j=2}^k$, and in particular $K_F/K$ is normal.

\medskip

Finally, the action of $G_F=\Gal(K_F/K)$ on $K_F$ extends to the polynomial ring $K_F[x_1,\ldots,x_k]$ by acting trivially on the variables $x_i$.  
Since
$
   F = \prod_j \ell_j^{a_j}
$
is $G_F$-invariant, and the factorization of $F$ is unique up to permutation of the $\ell_j$ and multiplication by units, it follows that $G_F$ permutes the linear forms $\ell_j$.  
Because we have normalized the forms to be monic in $x_1$, no nontrivial unit multiples occur, and hence $G_F$ acts by permuting the $\ell_j$ themselves.
\end{proof}\begin{corollary}
Assume $K$ is an infinite perfect field.  
Let 
\[
   F = \prod_i \ell_i^{a_i} \;\in\; K[x_1,\ldots,x_k]
\]
be a splitting form with distinct $\ell_i\in L[x_1,\ldots,x_k]$, for some field extension $L/K$.  Then the associated field extension $K_F/K$ is Galois. Moreover, if we define the reduced form
\[
   F^{\mathrm{red}} := \prod_i \ell_i \;\in\; K[x_1,\ldots,x_k],
\]
then the Galois groups coincide:
\[
   G_{F^{\mathrm{red}}} = G_F.
\]
\end{corollary}

\begin{proof} By Proposition \ref{unique} is normal, and by the assumption it is separable so it is Galois.
By construction, the Galois group $G_F = \Gal(K_F/K)$ permutes the linear factors $\ell_i$.  
Hence the product $\prod_i \ell_i = F^{\mathrm{red}}$ is invariant under $G_F$, and its coefficients lie in $K$.  
Thus the splitting field $K_{F^{\mathrm{red}}}$ coincides with $K_F$, and the corresponding Galois groups agree:
\[
   G_{F^{\mathrm{red}}} = \Gal(K_{F^{\mathrm{red}}}/K) \;=\; \Gal(K_F/K) \;=\; G_F.
\]
\end{proof}

\subsection{Normal crossings forms}
\begin{definition} Let $K$ be a field of characteristic zero, and   $\AA^n_K:=\Spec(K[x_1,\ldots,x_k])$ be the affine space with a possibly empty SNC divisor $E=\divv(x_{s+1}\cdot\ldots\cdot x_k)$,where $s\geq 0$. We say  that  $$
  H =x^{a_{s+1}}_{s+1}\cdot\ldots\cdot x^{a_k}_k\cdot F\in K[x_1,\ldots,x_k]
$$ is a normal crossing form centered at $V(x_1,\ldots,x_k)$ if $H$ and thus also $F$ are splitting forms and there is a field extension $K_H/K$ and presentation of \[
   H \;=\; \prod_{i=1}^{s} \ell_i(x_1,\ldots,x_k)^{a_i}\cdot x^{a_{s+1}}_{s+1}\cdot\ldots\cdot x^{a_k}_k
\in K_H[x_1,\ldots,x_k]\] such that the $K_H$ linear forms  $\ell_1,\ldots,\ell_s,x_{s+1}\ldots,x_k$ are linearly independent.

\end{definition} 

The following results are immediate:
\begin{lemma} With the above assumptions. Let   $$
  H =x^{a_{s+1}}_{s+1}\cdot\ldots\cdot x^{a_k}_k\cdot F\in K[x_1,\ldots,x_k]
$$ be  a normal crossing form at $V(x_1,\ldots,x_k)$ where $F$ is a splitting form. Let $K_H/K$ be its Galois extension. Then the Galois group  $G_H$ preserves   divisorial coordinates. Moreover we have $$K_H=K_F=K_{F^{red}}$$
\end{lemma}

\begin{proof} It is a consequence of Proposition \ref{unique}. One needs to show that $G_H$ preserves  divisorial coordinates
but one can assume that monic form (up to rescaling) is achieved by a coordinate change over the rational numbers such that $x_i=x_i+\alpha_{i1} x_1$ for divisorial coordinates $i=k-s+1,\ldots,k.$  In such case $G_H$ does not move divisorial coordinates occuring in the decomposition of $H$. Since $\alpha_{i1}\in \QQ$, where $i=k-s+1,\ldots,k,$ this also implies that
$K_H=K_F=K_{F^{\red}}$.
\end{proof}


\section{Canonical Centers and Initial Forms}

\subsection{Galois \'Etale Extensions}

Recall the following standard lemma:

\begin{lemma}
Let $R$ be an integral domain with field of fractions $K(R)$, and let $Z=\Spec(R)$ be the associated integral affine scheme.  

Let 
\[
f(y) = y^n + a_1y^{n-1} + \cdots + a_0 \in R[y]
\]
be a monic polynomial, and let $p \in Z$ be a point. Suppose that the reduction
\[
\overline{f}_p(y) := y^n + a_1(p)y^{n-1} + \cdots + a_0(p) \in \kappa(p)[y]
\]
has $n$ distinct roots in some algebraic extension of the residue field $\kappa(p) = R/m_p$.  

Then $f'(y) = \tfrac{df}{dy}$ is a unit in $R[z]/(f)$, and the morphism
\[
\pi \colon \Spec\bigl(R[y]/(f)\bigr) \longrightarrow Z=\Spec(R)
\]
is \'etale in a neighborhood of the fiber $\pi^{-1}(p)$.
\end{lemma}

\begin{proof}
The assumption that $\overline{f}_p(y)$ has no multiple roots means that $\gcd(\overline{f}_p(y), \overline{f}_p'(y)) = 1$ in $(R/m_p)[y]$.  
Equivalently, in the localized ring $R_p[y] = \mathcal{O}_{Z,p}[y]$, we have $(f(y), f'(y))=1$.  

Thus $f'(y)$ becomes a unit in $R_p[y]/(f)$, and the morphism
\[
R_p \;\longrightarrow\; R_p[y]/(f)
\]
is a standard \'etale extension. Therefore, $\pi$ is \'etale in a neighborhood of the fiber $\pi^{-1}(p)$.
\end{proof}
\subsection{Galois \'etale extensions of splitting forms}
\begin{corollary}\label{splitting}
Let $R$ be an integrally closed domain, and let $f_i(y) \in R[y]$ for $i=1,\ldots,k$ be monic polynomials.  
Let $L = K(R)(\alpha_{ij})$ be a finite extension of the fraction field $K(R)$ containing all the roots 
\[
\alpha_{i1}, \ldots, \alpha_{ik_i} \in L
\]
of the polynomials $f_i(y) \in R[y]$.  
Define the finitely generated $R$-subalgebra
\[
   S := R[\alpha_{ij}]_{i,j} \subset L.
\]  
This yields a natural finite morphism
\[
\varphi \colon Y=\Spec(S) \longrightarrow X=\Spec(R).
\]

If $q \in \Spec(S)$ with $p = \varphi(q) \in \Spec(R)$, and if the reductions
\[
\overline{\alpha}_{i1}(q), \ldots, \overline{\alpha}_{ik_i}(q)\quad \in \quad \kappa(q)=S/m_q
\]
are distinct roots of the reductions $\overline{f}_{i,p}(y) \,\in \kappa(p)[y]$, then $\varphi$ is finite \'etale over  a neighborhood of $p \in X$.
\end{corollary}

\begin{proof}
For each root $\alpha_{ij}\in L$, let $f_{ij}(y_{ij}) \in K(R)[y_{ij}]$ denote its minimal monic polynomial.  
Since $R$ is integrally closed, we have $f_{ij}(y_{ij}) \in R[y_{ij}]$.  

By the preceding lemma, the morphism
\[
W := \Spec\!\bigl(R[y_{ij}]_{i,j}/(f_{ij})\bigr) \;\longrightarrow\; X=\Spec(R)
\]
is finite and \'etale over a neighborhood of $p$.  
After shrinking $X$, we may assume $W \to X$ is \'etale and that $W$ is normal.  

Now consider the natural surjective homomorphism of $R$-algebras
\[
R[y_{ij}]_{i,j}/(f_{ij}) \;\longrightarrow\; S, 
\quad y_{ij} \longmapsto \alpha_{ij}.
\]
This shows that $Y = \Spec(S)$ is a closed subscheme of $W$ that dominates $X$.  
Hence $Y$ is an irreducible component of $W$.  
Since \'etaleness is preserved under passing to irreducible components, it follows that
\[
\Spec(S) \longrightarrow \Spec(R)
\]
is finite \'etale in a neighborhood of $p$.
\end{proof}\medskip

\begin{definition}
Let $R$ be an integral domain with fraction field $K(R)$.  Let $X=\Spec(R)$ be an irreducible smooth smooth affine variety over a field of characteristic zero, with SNC divisor $E_R$. 
Consider affine space $\AA^k_R=\Spec(R[x_1,\ldots,x_k])$.
Let
$
   F \in R[x_1,\ldots,x_k]
$
be a monic splitting form, admitting a factorization
$
   F = \prod_i \ell_i^{a_i}
$
into distinct monic linear forms
\[
   \ell_i(x) = \sum_{j=1}^k \alpha_{ij} x_j,
\]
where the coefficients $\alpha_{ij}$ lie in some field extension $L$ of $K(R)$.  

The \emph{splitting algebra} of $F$ is the finitely generated $R$-subalgebra
\[
   R_F := R[\alpha_{ij}]_{i,j} \subseteq L,
\]
and we set
\[
   X_F := \Spec(R_F).
\]

We refer to $R_F$ (or equivalently $X_F$) as the \emph{finite Galois  extension associated with $F$}.

\end{definition}
More generally  

\begin{lemma}
Let $p \in V(x_1,\ldots,x_k) \simeq X=\Spec(R)$, and let
\[
F \in R[x_1,\ldots,x_k]
\]
be a splitting form of degree $d$ over $K(R)$.
Assume that its reduction at $p$ satisfies $\overline{F}_p \neq 0$,
equivalently $\ord_p(F)=d$.

Then, after multiplying $F$ by a unit at $p$ and performing a suitable
change of coordinates, we may assume that $F$ is monic in a neighborhood
of $p$. Consequently, $F$ determines a unique finite Galois extension
\[
X_F=\Spec(R_F)\longrightarrow X=\Spec(R)
\]
in a neighborhood of $p$, together with the induced Galois cover
\[
\mathbb A^k_{X_F}
=\Spec\!\big(R_F[x_1,\ldots,x_k]\big)
\longrightarrow
\mathbb A^k_X
=\Spec\!\big(R[x_1,\ldots,x_k]\big).
\]
\end{lemma}
  \begin{proof}
After a generic change of coordinates, we may assume that $F_p$ is monic 
in $x_1$. Multiplying by an invertible function if necessary, we may further 
assume that $F$ is monic in $x_1$ in a neighborhood of $p\in X$.

Viewed as a monic polynomial over $K(R)$, the function $F$ splits as
\[
F=\prod_{i=1}^{s}\ell_i(x_1,\ldots,x_k)^{a_i},
\]
where each factor is monic of the form
\[
\ell_i(x)=x_1+\sum_{j=2}^{k}\alpha_{ij}x_j,
\]
after passing to the Galois extension $K(R)[\alpha_{ij}]$.

This construction induces a finite Galois extension $R[\alpha_{ij}]$
in a neighborhood of $p$.  The resulting extension is independent of
the chosen coordinates, by an argument analogous to the proof of
Proposition~\ref{unique}.

Fix a coordinate system $x_1,\ldots,x_k$.  Let
$x'_1,\ldots,x'_k$ be another linear coordinate system over $R$.
Then the coefficients $\beta_{ij}$ of the linear forms $\ell_i$
with respect to $x'_1,\ldots,x'_k$ are obtained from the coefficients
$\alpha_{ij}$ by an invertible matrix $A=(a_{ij})$ with entries in $R$.
Hence
\[
R[\alpha_{ij}]=R[\beta_{ij}].
\]

Although the forms $\ell_i$ may cease to be monic after the change of
coordinates, each $\ell_i$ can be rescaled by a unit
$\alpha\in R[\alpha_{ij}]$ so that
\[
\ell'_i=\alpha^{-1}\ell_i
\]
is monic in $x'_1$.  Reversing the coordinate change yields
$\ell_i=\beta\,\ell'_i$ with $\alpha\beta=1$ and
$\alpha,\beta\in R[\alpha_{ij}]$.
The coefficients of the monic forms $\ell'_i$ satisfy
\[
\alpha'_{ij}=\beta_{ij}/\alpha=\beta_{ij}\beta,
\]
and therefore
\[
R[\alpha'_{ij}]\subseteq R[\alpha_{ij}].
\]

By symmetry, the reverse inclusion holds, and thus
\[
R[\alpha_{ij}]=R[\alpha'_{ij}].
\]

\end{proof}
\subsection{Galois extensions and the canonical invariant}

\begin{proposition} \label{main} 
Let $X=\Spec(R)$ be an irreducible smooth smooth affine variety over a field of characteristic zero, with SNC divisor $E_R$. 
Consider affine space $\AA^k_X=\Spec(R[x_1,\ldots,x_k])$ with  SNC divisor $E=\divv(x_{s+1}\cdot\ldots\cdot x_k)\cup \pi^{-1}(E_R)$, where $\pi: \AA^k_X\to X$ be the standard projection.
Let
$$
  H =x^{a_{s+1}}_{s+1}\cdot\ldots\cdot x^{a_k}_k\cdot F\in R[x_1,\ldots,x_k]
$$ 
be a form of degree $d$ which is a Normal Crossings over $k(R)$. Consider $H$-admissible center $\cJ=(x_1,\ldots,x_k)^d$ on $\AA^k_R$. Note that  $V(\cJ)\subset \AA^k_X$ can be identified with $X$.
 Let  $p\in V(\cJ)$, such that $\ord_p(H)=d$.
Consider Galois extension, \[
  X_H=\Spec(R_H) \longrightarrow X=\Spec(R)=V(\cJ)
\] , and the induced Galois cover $$ \pi : \AA^k_{X_H}\to \AA^k_X$$ so that \[
   H= \ell_1^{a_1}\cdot\ldots\cdot \ell_s ^{a_s}\cdot x^{a_{s+1}}_{s+1}\cdot\ldots\cdot x^{a_k}_k\in R_H[x_1,\ldots,x_k].\]
 The following are equivalent:
\begin{enumerate}
\item 
$\inv_p(H) = (\underbrace{d,\ldots,d}_{s \text{ times}},\underbrace{d_+,\ldots,d_+}_{k-s \text{ times}})$.  

\item 
$\inv^1_p(H) = (\underbrace{d,\ldots,d}_{s \text{ times}},\underbrace{d_+,\ldots,d_+}_{k-s \text{ times}})$.

\item $\mathcal{J}=(x_1,\ldots,x_k)^d$ be a maximal admissible center of  $H$ at $p\in V(\cJ)$

\item For any $q \in \pi^{-1}(p)\subset X_F$
the linear forms $\ell_1,\ldots,\ell_s, x_{s+1},\ldots,x_k$ are linearly independent at $q$.  
\end{enumerate}
\end{proposition}

\begin{proof} $(1)\Leftrightarrow(3)\Rightarrow(2)$ follows directly from the definitions.

By Lemma~\ref{invar}, since $H \subset \mathcal J=(x_1,\ldots,x_k)^d$
and $\mathcal J$ is a maximal $H$-admissible center at $\eta_{\mathcal J}$,
where $H$ is NC, we obtain
\[
\inv(\mathcal J)
=
(\underbrace{d,\ldots,d}_{s\text{ times}},
 \underbrace{d_+,\ldots,d_+}_{k-s\text{ times}}).
\]
Hence $\mathcal J$ is maximal admissible at a point $p\in V(\cJ)$ if and only if
\[
\inv_p(H)=\inv(\mathcal J).
\]

$(2)\Rightarrow (4)$ Assume  $\inv^1_p(H)=\inv^1(\cJ)=(\underbrace{d,\ldots,d}_{s \text{ times}},\underbrace{d_+,\ldots,d_+}_{k-s \text{ times}})=\inv(\cJ)$. By  Lemma \ref{gr},
\[
   \inv^1_p(H) = \inv^1_p(\inn_p(H)) = \inv^1_p(\overline{H}_p).
\]
 Then  for $q\in \pi^{-1}(p)$  we have  the decomposition of the forms over the field $\kappa(q)$: $$\overline{H}_q= \overline{\ell}_{1q}^{a_1},\ldots,\overline{\ell}_{sq}^{a_s},\overline{x}_{s+1q}^{a_{s+1}}\cdot\ldots\cdot ,\overline{x}_{kq}^{a_k} $$ We deduce that  $\overline{\ell}_{1q},\ldots,\overline{\ell}_{sq},\overline{x}_{s+1q},\ldots,\overline{x}_{kq}$ are linearly independent at $q$, and $$\ell_1,\ldots,\ell_s, x_{s+1},\ldots,x_k$$ are linearly independent at $q$, otherwise $\overline{H}_p$ depended upon a smaller number of coordinates , those linearly independent among $$\overline{\ell}_{1q},\ldots,\overline{\ell}_{sq}, \overline{x}_{s+1q},\ldots, \overline{x}_{kq},$$ 
say $$\overline{\ell}_{1q},\ldots,\overline{\ell}_{s-1q}, \overline{x}_{s+1q},\ldots, \overline{x}_{kq},$$  so that  $\overline{H}_q\subset \cJ'=(\overline{\ell}_{1q},\ldots,\overline{\ell}_{s-1q}, \overline{x}_{s+1q},\ldots, \overline{x}_{kq})^d$,which implied that  that $\inv^1_p(H) = \inv^1_p(\overline{H}_p)\geq \inv(\cJ')=(\underbrace{d,\ldots,d}_{s-1 \text{ times}},\underbrace{d_+,\ldots,d_+}_{k-s \text{ times}})>\inv(\cJ)$.

$(4)\Rightarrow (2)$ Conversely if $\ell_1,\ldots,\ell_s, x_{s+1},\ldots, x_k$ are linearly independent at $q$ as well as their reductions over $\kappa(q)$ then $$\inv_p(H)\leq \inv_p^1(H)=\inv_q^1(H)=\inv_q^1(\overline{H}_q)=\inv^1(\cJ)= \inv(\cJ).$$
On the other since $H\subset J$, $\inv_p(H)\geq \inv(\cJ)$. So $\inv_p(H)= \inv(\cJ)$.

\end{proof}

\begin{proposition} \label{main0} Under the above assumptions
if the conditions  from the previous proposition hold, then:
\begin{enumerate}
\item The Galois extension morphism $\pi : X_H=X_F \to X=V(\cJ)$ is \emph{\'etale} at $p$, as well as the induced Galois cover  $\mathbb{A}^n_{X_H}\to \mathbb{A}^n_{X}$
\item The forms $H$ and  $F$ has normal crossings at $p$, and split into linear factors in the \'etale neighborhood $\mathbb{A}^n_{X_H}$.  
\end{enumerate}
Thus $F,H$ have  normal crossings at all points of $V(J)$ where $\cJ$ is maximal admissible for $\cJ$. In particular, there exists a Galois \'etale extension $X_H \to X$ such that 
\[
   H= F \cdot \prod_{i=s+1}^k x^{a_{i}}_{i}
=\prod_{i=1}^s \ell_i^{a_i}\cdot\prod_{i=s+1}^k x^{a_{i}}_{i}
\]
defines a simple normal crossing (SNC) divisor on  the induced Galois cover
\[
   \mathbb{A}^n_{X_H} \longrightarrow \mathbb{A}^n_X,
\]
with respect to the partial coordinate system  adapted to $E$ given by the linear forms $$\ell_1,\ldots,\ell_s,x_{s+1},\ldots,x_k$$ on $X_H$.
\end{proposition}

\begin{proof} Since $\ell_1,\ldots,\ell_s, x_{s+1},\ldots,x_k$ are linearly independent at $p$,
after a linear change of coordinates over $\mathbb Q$ preserving the
exceptional divisor coordinates $x_{s+1},\ldots,x_k$, we may assume that
\[
F^{\mathrm{red}}=\prod_{i=1}^s \ell_i
\]
is monic in $x_1$.

The associated finite Galois extension
\[
R_H=R_F=R_{F^{\mathrm{red}}}=R[\alpha_{ij}]
\]
is generated by the roots $\alpha_{ij}$, for
$i=1,\ldots,s$ and $j=2,\ldots,k$, of the monic polynomials
\[
\Phi_j
=
\prod_{i=1}^s (x_1-\alpha_{ij})
=
F^{\mathrm{red}}(x_1,\; x_j=1,\; x_{j'}=0 \text{ for } j'\neq 1,j)
\in R[x_1].
\]

For $\ell=2,\ldots,k$, consider generic linear combinations
\[
v_\ell=(v_{\ell i})_{i=1}^s,
\qquad
v_{\ell i}=\sum_{j=2}^k \beta_{\ell j}\alpha_{ij},
\]
where
\[
\beta_\ell=(\beta_{\ell2},\ldots,\beta_{\ell k})^T\in R^{k-1}.
\]
Equivalently,
\[
v_\ell = (\alpha_{ij})_{i,j}\,\beta_\ell .
\]

We will show that, for generic choices of $\beta_\ell$,
the elements $v_{\ell i}$ generate $R_H$ and their reductions at $p$
have minimal polynomials without repeated roots. Since
\[
\ell_1=\sum_j \alpha_{1j}x_j,\;\ldots,\;
\ell_s=\sum_j \alpha_{sj}x_j,
\quad x_{s+1},\ldots,x_k
\]
are linearly independent at $q$ over $p$, the rows of the extended matrix
\[
(\alpha_{ij})_{i=1,\ldots,s}^{j=1,\ldots,k}
\]
are linearly independent $(*)$.  In the first column ($j=1$) we have
\[
\alpha_{11}=\cdots=\alpha_{s1}=1 \qquad (**).
\]

For $\ell=2,\ldots,k$, the coordinates of
\[
v_\ell=(v_{\ell i})_{i=1}^s
\]
are roots of the specialization polynomial
\[
\Psi_\ell
=
\prod_{i=1}^s (x_1-v_{\ell i})
=
F^{\mathrm{red}}(x_1,x_2=\beta_{\ell2},\ldots,x_k=\beta_{\ell k})
\in R[x_1].
\]
Reducing modulo $p$, we obtain
\[
\overline{\Psi}_{\ell p}
=
\prod_{i=1}^s (x_1-\overline v_{\ell ip})
=
\overline F(x_1,x_2=\overline\beta_{\ell2p},\ldots,x_k=\overline\beta_{\ell kp})
\in \kappa(p)[x_1].
\]

The required conditions are:
\begin{itemize}
\item $\overline v_{\ell ip}\neq \overline v_{\ell i'p}$ for $i\neq i'$;
\item the vectors
\[
\overline v_{\ell p}
=
(\overline{\alpha}_{ijq})_{i=1,\ldots,s}^{j=2,\ldots,k}\,\overline\beta_\ell,
\qquad \ell=2,\ldots,k,
\]
generate the column space
\[
\operatorname{span}_{\kappa(p)}(\overline{\alpha}_j)_{j=2}^k,
\]
where $\overline{\alpha}_j=(\overline{\alpha}_{ijq})_{i=1}^s$.
\end{itemize}

These conditions are satisfied for $\overline\beta_\ell\in\kappa(p)^{k-1}$
outside finitely many proper linear subspaces defined by
\[
\sum_{j=2}^k
(\overline{\alpha}_{ijq}-\overline{\alpha}_{i'jq})
\,\overline{\beta}_{\ell jp}=0,
\]
and by the nonvanishing condition
\[
\det[\overline{\beta}_{\ell jp}]_{\ell,j=2}^k \neq 0.
\]

Hence $R_F=R[\alpha_{ij}]=R[v_{\ell i}]$, and the reductions
\[
\overline{\Psi}_{\ell p}
=
\prod_{i=1}^s (x_1-\overline v_{\ell i})
\]
have distinct roots.  By Corollary~\ref{splitting}, the Galois extension
$X_H=X_{F^{\mathrm{red}}}\to X$ is \'etale.  The remaining assertions follow
from Lemma~\ref{main}.

\end{proof}

\subsection{Galois group of  NC singularities}
\subsubsection{Weighted Normal Bundle to a Center}(\cite{ W22})

Given a center $\cJ$ on a smooth variety $X$, we can associate to it the graded algebra
\[
  \cO_X[\cJ t] = \cO_X \oplus \cJ t \oplus \cJ^2 t^2 \oplus \cdots,
\]
where the grading is defined by the filtration $\cJ^i$.  
The associated graded algebra 
\[
  \gr_{\cJ}(\cO_X) = \bigoplus_{i \ge 0} \cJ^i / \cJ^{i+1}
\]
is naturally isomorphic to the coordinate ring of the normal bundle $N_{X/\cJ}$ of $X$ along the center $\cJ$.

\begin{lemma}
Let $p \in X$ be a point, and suppose $\cJ$ is a canonical center for an ideal $I$ at $p$.  
Then $\cJ$ is also a canonical center for the initial ideal $\init_{\cJ}(I)$ on the normal bundle $N_{X/\cJ}$.  
In particular, the invariant associated to $I$ at $p$ coincides with that of the initial ideal:
\[
  \inv_{p}(I) = \inv_{p}\!\bigl(\init_{\cJ}(I)\bigr).
\]
\end{lemma}

\begin{proof}
The proof is by induction on the ambient dimension.  
Choose a hypersurface of maximal contact that works simultaneously for both $f$ and its initial form (this is always possible).  
With such a choice, the inductive definition of the invariant $\inv$ - expressed in terms of the coefficient ideal on the hypersurface of maximal contact - gives the same value in both cases.  
Thus the invariants coincide, and the canonicity of $\cJ$ is preserved for the initial ideal.
\end{proof}

\subsubsection{Galois \'Etale Splitting of the Initial Form}

\begin{lemma}\label{etale}
Let $\mathcal I$ be a coherent ideal on a smooth affine variety $X$ with SNC divisor $E$, 
and let $\mathcal J$ be a maximal $\mathcal I$-admissible center. 
Assume that $\mathcal I$ is normal crossing of codimension $1$ and order $d$ 
near the generic point $\eta_{\mathcal J}$.

Then $\mathcal I$ is locally principal along $V(\mathcal J)$. 
For any $p\in V(\mathcal J)$, write locally $\mathcal I=(h)$. 
Set
\[
H:=\mathrm{in}_{\mathcal J}(h)\in \mathrm{gr}_{\mathcal J}(\mathcal O_X).
\]
Then $H$ is normal crossing along $V(\mathcal J)$.

Moreover, after passing to a $\cJ$ split neighborhood (Lemma \ref{split})with projection $U\to V(\cJ)$ and to the associated Galois \'etale cover 
$U_H=U\times_{V(\cJ)}V_H \to U$ induced by the Galois extension $V_H\to V(\mathcal J)$, there exist local parameters 
$x_1,\dots,x_k$ centered at $\mathcal J$, adapted to $E$, such that $\mathcal J=(x_1,\dots,x_k)^d$ and that
\[
\qquad
H=\mathrm{in}_{\mathcal J}(h)=\mathrm{in}_{\mathcal J}(x_1)^{a_1}\cdots \mathrm{in}_{\mathcal J}(x_k)^{a_k}.
\]
is SNC on $N_\cJ(U_H)=N_\cJ(U)\times_{V(\cJ)}V_H$
\end{lemma}
\begin{proof}
Since $\mathcal I_{\eta_{\mathcal J}}$ is normal crossing of codimension~$1$ 
at $\eta_{\mathcal J}$, it is locally principal and generated by 
$h = x^\alpha f$. Let $(h)$ denote its schematic closure in a neighborhood 
of $p \in V(\mathcal J)$. Then
\[
(h) \supset \mathcal I,
\qquad \text{hence} \qquad
\mathcal I \subseteq (h) \subset \mathcal J.
\]
Thus $\mathcal J$ is maximal admissible for $(h)$.

Write $\mathcal I = (h)\mathcal I'$. Since
\[
\ord_p(\mathcal I) = \ord_p(h) = \ord_p(\mathcal J),
\]
we obtain $\ord_p(\mathcal I')=0$, hence 
$\mathcal I'=\mathcal O_X$ near $p$. 
Therefore $\mathcal I=(h)$ in a neighborhood of $p$.

Because $h$ has normal crossings at $\eta_{\mathcal J}$, 
there exists an \'etale neighborhood $U$ of $\eta_{\mathcal J}$ 
with local parameters $z_1,\dots,z_k$
adapted to $E$, such that
\[
h = z_1^{a_1}\cdots z_k^{a_k},
\qquad
\mathcal O_U\cdot \mathcal J=(z_1,\dots,z_k)^d.
\]

Passing to the graded algebra
\[
\gr_{\mathcal J}(\mathcal O_{U,\eta_{\mathcal J}})
\simeq \kappa(\eta_{\mathcal J})[y_1,\dots,y_k],
\]
we obtain
\[
H=\mathrm{in}_{\mathcal J}(h)
=\mathrm{in}_{\mathcal J}(z_1)^{a_1}\cdots
\mathrm{in}_{\mathcal J}(z_s)^{a_k},
\]
which is normal crossing over the {\it generic point $\eta_{\mathcal J}$}.

Since $\mathcal J$ is maximal admissible for $H$ 
at the same points where it is maximal for $\mathcal I$, 
Proposition~\ref{main0} and Lemma \ref{split} imply that after passing first to $\cJ$ split cover  with projection $U\to V(\cJ) $ and to the associated 
Galois \'etale cover $U_H=U\times_{V(\cJ)}{V(\cJ)}_H$,  $H$ defines an SNC divisor on $U_H$. 
More precisely, there exist local parameters 
$x_1,\dots,x_k$ on the corresponding cover $U_F$, 
centered at $V(\mathcal J)$ and adapted to the preimage of $E$, such that $ \cO_{U_F}\cdot\mathcal J=(x_1,\dots,x_k)^d,$
and \[
H=\inn_{\cJ}(H)=\inn_{\cJ}(x_1)^{a_1}\cdots \inn_{\cJ}(x_k)^{a_k}.
\]
\end{proof}\subsubsection{Galois Groups of Normal Crossing Singularities}

The construction above allows one to canonically associate a Galois group 
to any (embedded or non-embedded) normal crossing (NC) singularity.

\begin{definition}
Let $\mathcal I$ be a coherent ideal on a smooth affine variety $X$, 
and let $p\in V(\mathcal I)$ be a point where $\mathcal I$ is NC of codimension $k$.  
Let $\mathcal J$ be a maximal admissible center for $\mathcal I$ at $p$.

Then $\mathcal I$ contains the ideal $\mathcal I_Z$ of a smooth subvariety 
$Z\subset X$ of codimension $k-1$ such that 
$\mathcal I_{|Z}=(f)$ is locally principal and has NC of order $d$ at $p$.  
Let $\mathcal J_Z:=\mathcal J_{|Z}$ and denote by $\eta_{\mathcal J}$ 
its generic point.

The initial form
\[
F:=\mathrm{in}_{\mathcal J_Z}(f)
\in \mathrm{gr}_{\mathcal J_Z}(\mathcal O_Z)
\]
splits over a finite Galois extension 
$\kappa(\eta_{\mathcal J})_F/\kappa(\eta_{\mathcal J})$.

The \emph{Galois group of $\mathcal I$ at $p$} is
\[
\Gal(\mathcal I,p)
:=
\Gal\!\big(\kappa(\eta_{\mathcal J})_F/\kappa(\eta_{\mathcal J})\big).
\]
\end{definition}
\begin{definition}
Let $Y$ be a scheme of finite type over a field~$K$, 
and let $p \in Y$ be a point where $Y$ has NC of codimension~$k$.  
Consider a local closed embedding 
$Y \hookrightarrow X$ into a smooth variety $X$, 
defined by an ideal $\cI_Y \subset \cO_X$.  
Then the \emph{Galois group of $Y$ at $p$} is defined by
\[
\Gal(Y, p) := \Gal(\cI_Y, p).
\]
\end{definition}

\begin{lemma}
The Galois groups $\Gal(\mathcal{I}, p)$ and $\Gal(Y, p)$ are determined uniquely up to isomorphism. \qed
\end{lemma}

\begin{remark}
Normal crossing (NC) singularities need not become simple normal crossing (SNC) on their Galois cover;  
they are only \emph{formally} SNC, i.e., after completion along the maximal admissible center.
\end{remark}
\section{Factorization of Normal Crossing Singularities in {J}-Completion}

Let $\mathcal{I}$ be a coherent ideal on a smooth affine variety with SNC divisor $(X,E)$, and let 
$\mathcal{J} $ be its  maximal admissible center with divisorial $(x_{s+1}, \ldots, x_k)$ .  
Assume that $\mathcal{I}$ is normal crossing of codimension~$1$ in an open neighborhood of $V(\mathcal{J})$.

By Lemma~\ref{etale}, after passing to a suitable coordinate sytsem adapted to $E$ on an \'etale neighborhood of any point 
$p \in V(\mathcal{J})$, we may assume that:
\begin{itemize}
  
  \item $\cJ=(x_1, \ldots, x_k)$
  \item $\mathcal{I} = ( x_{s+1}^{a_{s+1} } \cdots x_k^{a_k} \cdot f)$ is principal, and
  \item $
   F = \operatorname{in}_{\mathcal{J}}(f)
     = x_1^{a_1} x_2^{a_2} \cdots x_s^{a_s} 
     \quad \text{in } \gr_\cJ(\cO_X)=\mathcal{O}_{V(\mathcal{J})}[x_1,\ldots,x_k],
\quad a_i \ge 1$. 
\end{itemize}

\begin{definition}
Let $R$ be an integral domain, let $k(R)$ denote its field of fractions, and let $\overline{k(R)}$ be an algebraic closure of $k(R)$.  
Consider a formal power series
\[
f \in R[[x_1, \ldots, x_k]].
\]
We introduce the following notions:
\begin{itemize}
    \item \textbf{Pre-SNC over R.}  
    The series $f$ is called \emph{pre-simple normal crossing (pre-SNC)} in $R[[x_1, \ldots, x_k]]$ if it can be written as
    \[
      f = x_1^{a_1} x_2^{a_2} \cdots x_s^{a_s} 
          + \sum_{|\alpha| > a_1+\ldots+a_k} c_{\alpha} x^{\alpha},
          \qquad c_{\alpha} \in R,
    \]
    for some integers $a_i \ge 1$.

    \smallskip
    \item \textbf{SNC over R.}  
    The series $f$ is called \emph{simple normal crossing (SNC)} in $R[[x_1, \ldots, x_k]]$ if it can be factorized as
    \[
       f = \prod_{i=1}^{s} (x_i + g_i(x))^{a_i},
       \qquad g_i(x) \in (x_1, \ldots, x_k)^2R[[x_1, \ldots, x_k]].
    \]

    \end{itemize}
\end{definition}
\medskip

\begin{lemma}\label{crucial}
Let $\mathcal{J} = (x_1, \ldots, x_k)^d$ be a maximal  admissible center for a function 
$f \in \mathcal{O}_X$, where $X$ is a smooth variety.  
Assume that
\[
   f = x_1^{a_1} x_2^{a_2} \cdots x_s^{a_s}
       + \sum_{\alpha} c_{\alpha} x^{\alpha}
   \ \in\ 
   \widehat{\mathcal{O}}_{X,\mathcal{J}}
   = \mathcal{O}_{V(\mathcal{J})}[[x_1, \ldots, x_k]]
\]
is simple normal crossing (SNC) in $\widehat{\mathcal{O}}_{X,\mathcal{J}}$.  
Then $f$ is normal crossing (NC) on $X$ along $\mathcal{J}$.
\end{lemma}

\begin{proof}
\footnote{We owe this argument to E.~Bierstone.}
If $f$ is SNC in the $\mathcal{J}$-adic completion $\widehat{\mathcal{O}}_{X,\mathcal{J}}$,  
then the corresponding equation $y^{\alpha} = f$ admits a solution in 
$\widehat{\mathcal{O}}_{X,\mathcal{J}}$ in a neighborhood of any point 
$p \in V(\mathcal{J})$.

By the Artin Approximation Theorem \cite[Theorem 1.10]{Artin}, such a formal solution can be approximated 
by an algebraic power series solution in the Henselization 
$\mathcal{O}^{h}_{X,\eta_{\mathcal{J}}} \subset \widehat{\mathcal{O}}_{X,\eta_{\mathcal{J}}}$ 
of $\mathcal{O}_{X,\eta_{\mathcal{J}}}$.  
Hence, after passing to an \'etale neighborhood of $p \in V(\mathcal{J})$,  
the same factorization holds in $\mathcal{O}_X$.  
Therefore, $f$ is NC on $X$ along $\mathcal{J}$.
\end{proof}

\begin{lemma}\label{generic}
Let $\mathcal{J} = (x_1, \ldots, x_k)^d$ be a maximal  admissible center for a function 
$f \in \mathcal{O}_X$, where $X$ is a smooth variety.  
Assume that $f$ admits a pre-SNC presentation
\[
   f = x_1^{a_1} x_2^{a_2} \cdots x_s^{a_s}
       + \sum_{\alpha} c_{\alpha} x^{\alpha}
       \in \widehat{\mathcal{O}}_{X,\mathcal{J}}
       = \mathcal{O}_{V(\mathcal{J})}[[x_1, \ldots, x_k]].
\]
If $(f)$ is normal crossing (NC) on $X$ at the generic point $\eta_{\mathcal{J}}$,  
then $f$ is SNC in
\[
   \widehat{\mathcal{O}}_{X,\eta_{\mathcal{J}}}
   \,\widehat{\otimes}\,
   \overline{\kappa(\eta_{\mathcal{J}})}=\overline{\kappa(\eta_{\mathcal{J}})}[[x_1, \ldots, x_k]],
\]
\end{lemma}

\begin{proof}
Suppose that $(f)$ is NC on $X$ at $\eta_{\mathcal{J}}$.  
Then there exists an \'etale neighborhood $U$ of $\eta_{\mathcal{J}}$ 
such that $f$ becomes SNC in $\mathcal{O}_U$.  
In particular, at the point $\eta_{\mathcal{J},U} \in U$ lying over $\eta_{\mathcal{J}}$, 
we have
\[
   \widehat{\mathcal{O}}_{U,\eta_{\mathcal{J},U}}
      = \kappa(\eta_{\mathcal{J},U})[[x_1, \ldots, x_k]].
\]

Let $\overline{\kappa(\eta_{\mathcal{J}})}
   = \overline{\kappa(\eta_{\mathcal{J},U})}$ 
be an algebraic closure of the residue field $\kappa(\eta_{\mathcal{J}})$ 
of the generic point of $\mathcal{J}$.  
Since $f$ is SNC on $U$, its factorization persists in the completion and extends
after tensoring with $\overline{\kappa(\eta_{\mathcal{J}})}$, giving
\[
   f = \prod_{i=1}^{k} (x_i + g_i(x))^{a_i}
   \quad \text{in } 
   \overline{\kappa(\eta_{\mathcal{J}})}[[x_1, \ldots, x_k]].
\]
Thus $f$ is SNC in 
$$\widehat{\mathcal{O}}_{X,\eta_{\mathcal{J}}}
   \widehat{\otimes} \overline{\kappa(\eta_{\mathcal{J}})}=\overline{\kappa(\eta_{\mathcal{J}})}[[x_1, \ldots, x_k]].$$
\end{proof}\medskip

\medskip

\subsection{Direct Algorithm for the Factorization of NC Singularities}
\begin{definition}[Residual order]
Let
\[
   f = x_1^{a_1} x_2^{a_2} \cdots x_s^{a_s}
       + \sum_{|\alpha| > d} c_\alpha x^\alpha
   \in R[[x_1, \ldots, x_k]],
\]
be a formal power series whose initial monomial
\[
   x_1^{a_1} x_2^{a_2} \cdots x_s^{a_s}
\]
has total degree
\[
   d := a_1 + \cdots + a_s.
\]
The \emph{residual order} of $f$, denoted $\mathrm{rord}(f)$,
is the order of the higher-degree part:
\[
   \mathrm{rord}(f)
   :=
   \ord\!\Big(\sum_{|\alpha| > d} c_\alpha x^\alpha\Big)
   =
   \min\{\,|\alpha| \mid c_\alpha \neq 0,\ |\alpha|>d\,\}.
\]
If $f = x_1^{a_1} x_2^{a_2} \cdots x_s^{a_s}$ is a pure monomial,
we set $\mathrm{rord}(f)=\infty$.
\end{definition}
\begin{definition}[SNC-resolvability of degree \(e\)]
Let
\[
   f = x_1^{a_1} x_2^{a_2} \cdots x_s^{a_s}
       + \sum_{|\alpha| \ge e} c_\alpha x^\alpha
   \in R[[x_1, \ldots, x_k]],
\]
be a formal power series of order
\[
   d := a_1 + \cdots + a_s,
\]
and assume that the residual order is finite, equal to \(e>d\).

We say that \(f\) is \emph{SNC-resolvable of degree \(e\) over \(R\)}  
if there exists a monomial \(x^\alpha\) of degree \(|\alpha|=e\)
with \(c_\alpha\neq 0\) such that, for some index
\(j\in\{1,\ldots,s\}\),
\[
   x_1^{a_1}\cdots x_j^{a_j-1}\cdots x_s^{a_s}
   \mid x^\alpha.
\]

If \(f = x_1^{a_1} \cdots x_s^{a_s}\) is a pure monomial,
we say that \(f\) is resolvable of degree \(\infty\).

The set
\[
   M_e(f)
   :=
   \bigl\{
      x^\alpha \;\big|\;
      |\alpha|=e,\;
      c_\alpha\neq 0,\;
      x_1^{a_1}\cdots x_j^{a_j-1}\cdots x_s^{a_s}
      \mid x^\alpha
      \text{ for some } j
   \bigr\}
\]
is called the \emph{minimal set of degree \(e\)} of \(f\).
It depends on \(f\) and on the chosen coordinate system
\((x_1,\ldots,x_k)\).
\end{definition}
\begin{lemma}\label{lem:snc-resolvable}
Let $R$ be an integral domain of characteristic~$0$.  
Every series
\[
   f = \prod_{i=1}^k (x_i + g_i)^{a_i}
   \in R[[x_1,\ldots,x_k]],
\]
which is simple normal crossing (SNC) over $R$,
is SNC-resolvable of degree $\rord(f)$.
\end{lemma}

\begin{proof}
Assume $\rord(f)<\infty$ and let
\[
d_0 := \min_i \ord(g_i), 
\qquad
d := \sum_i a_i.
\]
Write $g_i = G_i + (\text{higher order})$, where $G_i$ is homogeneous of degree $d_0$.

Expanding $f$ to the lowest nontrivial degree gives
\[
f = x^\alpha + \sum_i a_i\, x^{\check\alpha_i} G_i + \cdots,
\]
where $x^\alpha = \prod x_i^{a_i}$ and
$\deg(x^{\check\alpha_i} G_i)=d_0+d-1$.

If $d_0+d-1<\rord(f)$, the homogeneous part of that degree must vanish:
\[
\sum_i a_i\, x^{\check\alpha_i} G_i = 0.
\]
Hence $x_i \mid G_i$ for all $i$, so $G_i=x_iH_i$.
Factoring $(1-H_i)$ from each term increases $\ord(g_i)$ without changing $f$.

Repeating this process strictly raises $\min\ord(g_i)$.
After finitely many steps we reach
\[
\min\ord(g_i)+d-1=\rord(f).
\]

At this stage the homogeneous term of degree $\rord(f)$
is
\[
\sum_i a_i\, x^{\check\alpha_i} G_i,
\]
and every monomial in it is divisible by some $x^{\check\alpha_i}$.
Thus $f$ is SNC-resolvable of degree $\rord(f)$.
\end{proof}\begin{lemma}[Base change for SNC--resolvability]
\label{lem:SNC-base-change}
Let \( R \) be an integral domain with field of fractions \( k(R) \).  
Consider a power series
\[
   f
   = x_1^{a_1} x_2^{a_2} \cdots x_s^{a_s}
     + \sum_{\alpha} c_\alpha x^\alpha
   \in R[[x_1, \ldots, x_k]].
\]
If \( f \) is SNC--resolvable of some degree over the algebraic closure
\( \overline{k(R)} \) (that is, as an element of
\( \overline{k(R)}[[x_1, \ldots, x_k]] \)),
then \( f \) is already SNC--resolvable in \( R[[x_1, \ldots, x_k]] \).
\end{lemma}
\begin{proof}
SNC--resolvability depends only on  the divisibility relations among monomials in the presentation of \(f\).  
Since these properties are preserved under base change to 
\(\overline{k(R)}\), they already hold over \(R\).
\end{proof}

\subsection{Direct Algorithm for SNC Factorization}

Let $R$ be an integral domain of characteristic zero and
\[
  f = x_1^{a_1}\cdots x_k^{a_k} + \sum c_\alpha x^\alpha
  \in R[[x_1,\ldots,x_k]]
\]
be pre-SNC with residual order $e=\rord(f)$.
Assume that $f$ is SNC over $\overline{k(R)}$.

\medskip
\noindent
\textbf{Algorithm.}
Since $f$ is SNC over $\overline{k(R)}$, it is SNC-resolvable over $R$.
Choose a monomial $c_\alpha x^\alpha$ of minimal degree $|\alpha|=e$.
Select $j$ such that
\[
x_1^{a_1}\cdots x_j^{a_j-1}\cdots x_s^{a_s} \mid x^\alpha,
\]
and perform the coordinate change
\[
x_j \longmapsto
x_j + \frac{1}{a_j} c_\alpha
\frac{x^\alpha}{x_1^{a_1}\cdots x_j^{a_j-1}\cdots x_s^{a_s}}.
\]
This removes the chosen term without introducing lower-degree terms.

\medskip
After each substitution either the minimal set $M_e(f)$ shrinks
or the residual order increases.
Thus the invariant $(e,|M_e(f)|)$ decreases lexicographically.

Repeating the process produces transformations
\[
x_i \mapsto x_i + g_i^{(t)}(x),
\quad \ord(g_i^{(t)}) \to \infty.
\]
In the $(x)$-adic topology this converges formally to
\[
x_i \mapsto x_i + g_i(x),
\qquad \ord(g_i)\ge2.
\]

\medskip
\noindent
\textbf{Output.}
In the limit,
\[
f=\prod_{i=1}^s (x_i+g_i(x))^{a_i}
\]
in $R[[x_1,\ldots,x_k]]$.\begin{proposition}\label{SNC}
Let $R$ be an integral domain of characteristic zero.
If
\[
f = x_1^{a_1}\cdots x_s^{a_s}
    + \sum_{|\alpha|>d} c_\alpha x^\alpha
    \in R[[x_1,\ldots,x_k]]
\]
is SNC over $\overline{k(R)}$,
then $f$ is SNC over $R$.
\end{proposition}

\begin{proof}
Apply the above algorithm.
The divisibility conditions defining SNC-resolvability
are algebraic and preserved under base change.
Hence SNC over $\overline{k(R)}$ implies
SNC-resolvability over $R$.
The formal limit yields an SNC factorization over $R$.
\end{proof}
\begin{corollary}\label{main10}
Let \(X\) be a smooth variety over a field of characteristic zero,
and let \(E\) be an SNC divisor on \(X\).
Let \(h \in \mathcal{O}(X)\) and let \(\mathcal{J}\) be its maximal admissible center.
Assume that \(h\) is generically normal crossings (NC) along \(\mathcal{J}\),
and that the initial form \(\inn_\cJ(h)\) is NC along \(V(\mathcal{J})\).
Then \(h\) is NC along \(\mathcal{J}\).
\end{corollary}

\begin{proof}
Let \(\cJ = (x_1,\ldots,x_k)^d\) be the maximal admissible center of
\(h = x^\alpha f\).
By Lemma~\ref{etale}, the initial form
\(H=\inn_\cJ(h)=x^\alpha F\)
is SNC along \(V(\mathcal{J})\) on an \'etale Galois cover \(X_H\).
Hence \(f\) (and thus \(h\)) admits a pre-SNC presentation in
\[
\widehat{\mathcal{O}}_{X_H,\mathcal{J}}
= \mathcal{O}_{V(\mathcal{J})}[[x_1,\ldots,x_k]].
\]

By Lemma~\ref{generic}, \(h=x^\alpha f\) and \(f\) are SNC in
\[
\widehat{\mathcal{O}}_{X_H,\eta_{\mathcal{J}}}
\widehat{\otimes}
\overline{\kappa(\eta_{\mathcal{J}})}
=
\overline{\kappa(\eta_{\mathcal{J}})}[[x_1,\ldots,x_k]].
\]
Then Proposition~\ref{SNC} implies that \(h\) and \(f\) are SNC in
\(\mathcal{O}_{V(\mathcal{J})}[[x_1,\ldots,x_k]]\).
Finally, by Lemma~\ref{crucial}, this implies that \(h\) is NC along \(\mathcal{J}\).
\end{proof}More precisely applying the above arguments we have direct enhancing of Lemma \ref{etale}:
\begin{proposition}
Let $\mathcal I=(h)$ be a locally principal ideal on a smooth variety $X$
with SNC divisor $E$, and let $\mathcal J$ be an admissible center for $h$.
Assume that $\mathcal J$ is maximal admissible at the generic point
$\eta_{\mathcal J}$ of $V(\mathcal J)$ and that $h$ is NC of order $d$
at $\eta_{\mathcal J}$.
Let $H:=\mathrm{in}_{\mathcal J}(h)$, and let
$\pi:U_H\to U$ be the associated finite Galois cover
as in Lemma~\ref{etale}.

Then:
\begin{enumerate}
\item $h$ is NC at all points of $V(\mathcal J)$
where $\mathcal J$ is maximal admissible.
\item $\pi$ is \'etale over those points.
\item On $U_H$, the function $h$ is formally SNC along
$\mathcal J_H:=\mathcal O_{U_H}\cdot\mathcal J$,
i.e.\ in suitable formal parameters adapted to $E_H$,
\[
h = u_1^{a_1}\cdots u_k^{a_k},
\qquad
\mathcal J_H=(u_1,\ldots,u_k)^d.
\]
\item By Artin approximation, $h$ is SNC along $\mathcal J$
in an \'etale neighborhood.
\end{enumerate}
\end{proposition}
\begin{theorem}[Center Theorem~\ref{center}]\label{NCC}
Let $\mathcal I$ be an ideal on a smooth affine variety $X$
with an SNC divisor $E$, and let $\mathcal J$ be a maximal
admissible center for $\mathcal I$.
Assume that $\mathcal I$ is normal crossings adapted to $E$
in a neighborhood of the generic point of $\mathcal J$.
Then $\mathcal I$ is normal crossings adapted to $E$
along the center $\mathcal J$.
\end{theorem}\begin{proof}
By Lemma~\ref{etale2}, we have
\[
\inv_{\eta_{\mathcal J}}(\mathcal I)
=
\inv(\mathcal J)
=
(\underbrace{1,\ldots,1}_{r},
 \underbrace{d,\ldots,d}_{t},
 \underbrace{d_+,\ldots,d_+}_{k-r-t}).
\]

For any point $p\in V(\mathcal J)$ with
$\inv_p(\mathcal I)=\inv(\mathcal J)$,
the cotangent ideal $T(\mathcal I)$ contains
$r$ independent local parameters
$y_1,\ldots,y_r\in \mathcal I\subset \mathcal J$
forming a maximal contact.
Hence $\mathcal I$ contains the ideal
\[
\mathcal I_Z=(y_1,\ldots,y_r),
\]
defining a smooth subvariety $Z\subset X$.

By Lemma~\ref{etale2}, $\mathcal I$ is NC of codimension $r+1$
at $\eta_{\mathcal J}$ if and only if
$\mathcal I_{|Z}$ is NC of codimension $1$
at $\eta_{\mathcal J}\in Z$.
By Lemma~\ref{main10}, $\mathcal I_{|Z}$ is NC along
$V(\mathcal J)=V({\mathcal J|Z})$.
Applying Lemma~\ref{etale2} again, we conclude that
$\mathcal I$ is NC of codimension $r+1$ along $V(\mathcal J)$.
\end{proof}

\begin{example}
Consider the triple point
\[
  f = xyz + x^4 + y^4 + z^4.
\]
This singularity is \emph{not} NC: it is no SNC-\emph{resolvable} of order \(4\) as a power series.  
None of the minimal monomials \(x^4, y^4, z^4\) is divisible by any of the monomials \(xy, yz, xz\) appearing from the initial term.

\medskip
In contrast, any singularity of the form
\[
  f = xy + f(x,y), \qquad \ord_0(f) \geq 3,
\]

for instance

$$f=xy+x^3+y^3$$
\emph{is} NC, since it is always SNC-resolvable  in $\widehat{X}_0$ of order \(\rord_0(f)\).  
In this case, every monomial in \(f(x,y)\) occuring in the algorithm of SNC factorization is divisible by either \(x\) or \(y\), and hence can be eliminated by successive coordinate changes. 
Note that the algorithm and its factorization is only unique up to units. So the direct factorization algorithm may give us different results depending upon choices we made.
\end{example}

\section{Proofs of the main desingularization theorems}
\subsection{Proof of the Principalization Theorem~\ref{thm:princ-NC}}

\begin{proof}
Let $X_*^{\mathrm{nc}}(\mathcal I,E)\subset X$ denote the locus where
$\mathcal I$ is NC adapted to the divisor $E$.
If $\mathcal I$ is not NC, then
\[
Z:=X\setminus X_*^{\mathrm{nc}}(\mathcal I,E)
\]
is a nonempty closed subset.

The invariant $\inv$ is upper semicontinuous and attains its maximum on $Z$.
Let $\mathcal J$ be the maximal admissible center associated with this
maximal value of $\inv$.
By the Center Theorem~\ref{center},
\[
V(\mathcal J)\subset Z.
\]

Let $\sigma_1\colon X_1\to X$ be the weighted blow-up of $\mathcal J$,
and let $\mathcal I_1:=\sigma_1^c(\mathcal I)$ be the controlled transform.
Then the invariant $\inv$ strictly decreases over
\[
Z_1:=X_1\setminus X_{1*}^{\mathrm{nc}}(\mathcal I_1,E).
\]
Repeating the process, we eventually obtain $Z_k=\emptyset$,
so that $\mathcal I_k$ is NC on $X_k$.

If initially $\mathcal I=\mathcal I_F$ is an ideal defining
an NC divisor of codimension one (reduced or nonreduced) on an open set,
we apply the algorithm relative to $X_*^{\mathrm{nc}}(\mathcal I,E)$.
In the final stage, the total transform decomposes as
\[
\mathcal O_{X_k}\cdot\mathcal I
=
\mathcal I_D\cdot\mathcal I_k,
\]
where $\mathcal I_D$ defines the resulting NC divisor.
\end{proof}
\subsection{Proof of the Embedded NC-Preserving Resolution Theorem~\ref{thm:embedded-NC}}

\begin{proof}
Let $Z\subset X$ be a closed subscheme of a smooth variety $X$
with an SNC divisor $E$.
Consider the ideal $\mathcal I_Z$ and apply the principalization
algorithm described above to obtain a morphism
\[
\Pi\colon Y\to X.
\]
In this setting, one may use the controlled, weak, or strict transform
of the ideal.

The strict transform of $Z$ on $Y$ is then an NC subscheme
which has normal crossings with the SNC divisor
\[
D\cup \Pi^{s}(E),
\]
where $D$ denotes the exceptional divisor.
\end{proof}

\subsection{Proof of the Non-Embedded NC-Preserving Resolution Theorem~\ref{thm:nonembedded-NC}}

We follow the argument of \cite{W22}.  
A scheme of finite type over a field $k$ can be locally embedded as a
closed subscheme $Y\subset X=\mathbb A^n$ of a smooth variety.
Similarly, a locally affine Deligne--Mumford stack with presentation
$Y=[U/G]$, where $G$ is a finite group, admits a $G$-equivariant
closed immersion $U\hookrightarrow \mathbb A^n$ via a finite-dimensional
$k$-representation of $G$.  Hence $Y$ embeds into the smooth stack
$X=[\mathbb A^n/G]$.

This reduction allows us to apply the embedded
NC-preserving resolution (Theorem~\ref{thm:embedded-NC}).
Different local embeddings yield canonically isomorphic embedded
resolutions.  Indeed, two embeddings are locally related by
\'etale morphisms and closed immersions
\[
Y\hookrightarrow X_1\hookrightarrow X_2,
\]
where $\mathcal I_{X_1}\subset \mathcal I_Y$ corresponds to a
maximal contact.  In this situation,
\[
\inv_p(\mathcal I_{Y,X_2})
=
(1,\ldots,1,\inv_p(\mathcal I_{Y,X_1})).
\]
Thus, by functoriality, the invariant differs only by
leading entries equal to $1$ and is therefore independent
of the embedding.

For a fixed embedding $Y\subset X$ with $\dim X=n$,
write $\inv_p(\mathcal I_Y)=(b_1,\ldots,b_k)$.
Define $\widetilde{\inv}_p(Y)$ to be the equivalence class of
sequences $(b_1,\ldots,b_k)_n$, where
\[
(b_1,\ldots,b_k)_n
\sim
(\underbrace{1,\ldots,1}_{m},b_1,\ldots,b_k)_{n+m}.
\]
This invariant is functorial and independent of the chosen embedding.
Consequently, the non-embedded resolution, defined via the invariant
$\widetilde{\inv}$, is canonical at the level of strict transforms.
\subsection{Proof of the NC-Completion Theorem~\ref{NC-c}}

By Rydh's compactification theorem~\cite{Ryd14}, any separated
Deligne--Mumford stack $X$, in particular one with normal crossing
singularities, admits a compactification. That is, $X$ embeds as an
open dense substack of a proper Deligne--Mumford stack
\[
X \subset \overline{X}.
\]
Applying the non-embedded NC-preserving resolution
(Theorem~\ref{thm:nonembedded-NC}) to $\overline{X}$,
we obtain a proper Deligne--Mumford stack with normal crossings,
yielding the desired NC-compactification.\section{Examples-Cyclic Points of Bierstone--Milman--Pacheco}

\subsection{Cyclic and Ramified Singularities}
Recall that the pinch point
\[
x^2 - y^2 z = 0
\]
splits after passing to the finite ramified Galois extension
\[
x^2 - y^2 z = (x + \sqrt{z}\,y)(x - \sqrt{z}\,y),
\]
with Galois group $\mathbb Z_2$.  

This phenomenon motivated Bierstone--Milman--Pacheco~\cite{BDMP14} 
to introduce the class of cyclic singularities $\mathrm{cp}(n)$, 
which arise naturally in resolution processes (except in the NC case) 
and serve as fundamental examples of minimal singularities.

Such singularities are governed by their splitting forms coming from 
ramified Galois extensions of normal crossings singularities at maximal 
admissible centers.  
For instance, for $F=x^2-y^2z$, the center $(x,y)^2$ is maximal admissible 
on $z\neq 0$ but not at the origin.  
The corresponding extension
\[
K[x,y,z]\subset K[x,y,z^{1/2}]
\]
ramifies precisely at the pinch points, where the linear factors become 
linearly dependent.

The ramification locus cannot be removed by smooth blow-ups or ordinary 
birational transformations, but it can be resolved by suitable weighted 
blow-ups.  
For example, the weighted blow-up with weights $(2,3,3)$ eliminates the 
pinch point, whereas smooth blow-ups leave its local form unchanged.  
In particular, resolution by smooth centers eventually forces a center 
intersecting the NC locus, thereby altering it, while weighted admissible 
centers eliminate the non-NC singularity without modifying the NC locus.

Every minimal singularity of this type is controlled by its associated 
splitting form arising from a ramified Galois extension.
\begin{definition}[Bierstone--Milman--Pacheco~\cite{BDMP14}]
The \emph{cyclic singularity} $\mathrm{cp}(n)$ of order $n$ is defined by
\[
\Delta_n(x_0,\ldots,x_{n-1},z)
= \prod_{\ell=0}^{n-1}
\Bigl(
x_0 + \epsilon^{\ell} z^{1/n}x_1 + \cdots 
+ \epsilon^{(n-1)\ell} z^{(n-1)/n}x_{n-1}
\Bigr),
\]
where $\epsilon=e^{2\pi i/n}$.
\end{definition}

The associated Galois extension
\[
K(z)\subset K(z^{1/n})
\]
has group $\Gal(K(z^{1/n})/K(z))\simeq\mathbb Z_n$, 
generated by $z^{1/n}\mapsto \epsilon z^{1/n}$. 
This cyclic action permutes the linear factors of $\Delta_n$.

The corresponding splitting cover is
\[
X_F=\Spec_X(\mathcal O_X[z^{1/n}]).
\]

The weighted center $\mathcal J=(x_0,\ldots,x_{n-1})^n$ 
is admissible for $\Delta_n$ precisely along the NC locus
\[
V(\mathcal J)\setminus V(x_0,\ldots,x_{n-1},z),
\]
where the linear forms are independent.

\bibliographystyle{amsalpha}


\end{document}